\documentclass[11pt,letterpaper]{amsart}

\makeatletter
\def\subsection{\@startsection{subsection}{2}%
  \z@{.7\linespacing\@plus.7\linespacing}{.2\linespacing}%
  {\centering\normalfont\scshape}}
\makeatother

\usepackage[letterpaper,top=3.8cm, bottom=3.8cm, left=3.7cm, right=3.7cm]{geometry}

\usepackage{amssymb}

\newtheorem{theorem}{Theorem}[section]
\newtheorem{proposition}[theorem]{Proposition}

\newtheorem{lemma}[theorem]{Lemma}

\numberwithin{equation}{section}

\newcommand{\abs}[1]{\lvert#1\rvert}

\newcommand{\floor}[1]{\lfloor#1\rfloor}
\newcommand{\ang}[1]{\langle#1\rangle}

\newcommand{\e}{\epsilon}
\renewcommand{\d}{\delta}
\renewcommand{\t}{\tau}
\DeclareMathOperator{\Tr}{Tr}
\DeclareMathOperator{\tr}{tr}
\DeclareMathOperator{\rank}{rank}
\DeclareMathOperator{\GL}{GL}
\DeclareMathOperator{\wt}{wt}

\newcommand{\F}{\mathbb{F}}
\newcommand{\Z}{\mathbb{Z}}

\newcommand{\C}{\mathcal{C}}
\renewcommand{\L}{\mathcal{L}}

\begin{document}

\title[Symmetric bilinear forms over finite fields]{Symmetric bilinear forms over finite fields\\with applications to coding theory}

\author{Kai-Uwe Schmidt}
\address{Faculty of Mathematics, Otto-von-Guericke University, Universit\"atsplatz~2, 39106 Magdeburg, Germany}
\email{kaiuwe.schmidt@ovgu.de}

\date{25 October 2014}


\begin{abstract}
Let $q$ be an odd prime power and let $X(m,q)$ be the set of symmetric bilinear forms on an $m$-dimensional vector space over $\F_q$. The partition of $X(m,q)$ induced by the action of the general linear group gives rise to a commutative translation association scheme. We give explicit expressions for the eigenvalues of this scheme in terms of linear combinations of generalised Krawtchouk polynomials. We then study $d$-codes in this scheme, namely subsets $Y$ of $X(m,q)$ with the property that, for all distinct $A,B\in Y$, the rank of $A-B$ is at least $d$. We prove bounds on the size of a $d$-code and show that, under certain conditions, the inner distribution of a $d$-code is determined by its parameters. Constructions of $d$-codes are given, which are optimal among the $d$-codes that are subgroups of $X(m,q)$. Finally, with every subset~$Y$ of $X(m,q)$, we associate two classical codes over $\F_q$ and show that their Hamming distance enumerators can be expressed in terms of the inner distribution of $Y$. As an example, we obtain the distance enumerators of certain cyclic codes, for which many special cases have been previously obtained using long ad hoc calculations.
\end{abstract}

\maketitle

\section{Introduction}

Let $q$ be an odd prime power and let $X(m,q)$ be the set of symmetric bilinear forms on an $m$-dimensional vector space over $\F_q$. Note that we can identify $X(m,q)$ with the set of $m\times m$ symmetric matrices over $\F_q$. In this paper, we are mainly concerned with $d$-codes in $X(m,q)$, namely subsets $Y$ of $X(m,q)$ with the property that, for all distinct $A,B\in Y$, the rank of $A-B$ is at least $d$. In particular, for given~$m$ and $d$, we are interested in $d$-codes containing as many elements as possible. One of our motivations for this problem is a link between such $d$-codes and classical coding theory.
\par
The corresponding problem for finite fields of characteristic two has been studied in~\cite{Sch2010}, by exploiting close relationships between symmetric and alternating bilinear forms over finite fields of characteristic two and by building heavily on results for the association scheme of alternating bilinear forms~\cite{DelGoe1975}. The main tool of this paper is the association scheme of symmetric bilinear forms, which has been introduced in~\cite{HuoWan1993} and has been studied further in~\cite{WanWanMaMa2003}. This association scheme differs considerably from the classical association schemes used in coding theory, namely the Hamming scheme~\cite{Del1973} and its $q$-analog (or the association scheme of bilinear forms)~\cite{Del1978}, the Johnson scheme~\cite{Del1973} and its $q$-analog~\cite{Del1978a}, and the association scheme of alternating bilinear forms~\cite{DelGoe1975}. While these classical schemes are symmetric ($P$- and $Q$-) polynomial schemes, the association scheme of symmetric bilinear forms does not have such polynomial properties and is, in general, not symmetric.
\par
The association scheme of symmetric bilinear forms is reviewed in Section~\ref{sec:scheme}. Our key advance is the computation of its eigenvalues (also called the $P$- and $Q$-numbers). Although this association scheme is not polynomial and so the eigenvalues do not simply arise from evaluations of sets of orthogonal polynomials (see~\cite{Del1976} and~\cite{Sta1981} for the classical association schemes used in coding theory), they can still be written as linear combinations of evaluations of generalised Krawtchouk polynomials.
\par
In Section~\ref{sec:properties}, we use the knowledge of the eigenvalues to obtain bounds on the size of $d$-codes in $X(m,q)$. In particular, we prove that every $d$-code $Y$ that is an additive subgroup of $X(m,q)$ satisfies
\begin{equation}
\abs{Y}\le
\begin{cases}
q^{m(m-d+2)/2}     & \text{for $m-d$ even}\\[2ex]
q^{(m+1)(m-d+1)/2} & \text{for $m-d$ odd}.
\end{cases}   \label{eqn:bound}
\end{equation}
In Section~\ref{sec:constructions}, we provide constructions of subgroups of $X(m,q)$ that satisfy these bounds with equality. In the case that $d$ is odd (which appears to be the simpler situation), we prove that the bound~\eqref{eqn:bound} also holds for $d$-codes that are not necessarily subgroups of $X(m,q)$ and that, in case of equality in~\eqref{eqn:bound}, the inner distribution of~$Y$ is uniquely determined. We also give explicit expressions for the inner distribution in this case. The situation is quite different for even $d$. If $d$ is even and equality holds in~\eqref{eqn:bound}, then the inner distribution of $Y$ is, in general, not uniquely determined. However we give explicit expressions for the inner distributions of our constructions. In the case that $d$ is even, it is also an open question as to whether there exist $d$-codes~$Y$ in $X(m,q)$ that violate the bound~\eqref{eqn:bound}, and so are better than $d$-codes that are subgroups of $X(m,q)$. 
\par
In Section~\ref{sec:coding} we apply our results to classical coding theory. With every subset~$Y$ of $X(m,q)$ we associate two codes over $\F_q$ of length $q^m-1$ and show that the distance enumerators of these codes can be expressed in terms of the inner distribution of~$Y$. This provides a general approach to establish the distance enumerator of codes derived from symmetric bilinear (or quadratic) forms. For example, we immediately obtain the distance enumerator of a family of cyclic codes corresponding to the constructions of Section~\ref{sec:constructions}, for which many special cases have been previously obtained ~\cite{FenLuo2008},~\cite{LuoFen2008},~\cite{LiuLuo2014},~\cite{LiuLuo2014a},~\cite{LiuHan2013},~\cite{ZheWanZenHu2013},~\cite{LiuYanLiu2014} using long ad hoc calculations.
\par
It should be noted that the problems studied in this paper have natural analogs involving unrestricted bilinear forms and alternating bilinear forms, which have been studied in~\cite{Del1978} and~\cite{DelGoe1975}, respectively. These theories also have interesting applications in classical coding theory. As mentioned above, the main difference lies in the structure of the underlying association schemes, which makes the analysis considerably more difficult in the case of symmetric bilinear forms. It should also be noted that related, but different, rank properties of sets of symmetric bilinear forms have been studied in~\cite{DumGowShe2011}.

%%%%%%%%%%%%%%%%%%%%%%%%%%%%%%%%%%%%%%%%%%%%%%%%%%%%%%%%%%%%%%%%%%%%%

\section{The association scheme of symmetric bilinear forms}
\label{sec:scheme}

We begin with recalling some well known facts about association schemes in general and about the association scheme of symmetric bilinear forms in particular. For more background on association schemes and connections to coding theory we refer to~\cite{Del1973},~\cite{DelLev1998}, and~\cite{MarTan2009} and to~\cite[Chapter 21]{MacSlo1977} and~\cite[Chapter 30]{vLiWil2001} for gentle introductions. For more background on the association scheme of symmetric bilinear forms we refer to~\cite{HuoWan1993} and~\cite{WanWanMaMa2003}.
\par
An \emph{association scheme} with $n$ classes is a pair $(X,(R_i))$, where $X$ is a finite set and $R_0,R_1,\dots,R_n$ is a partition of $X\times X$, satisfying
\begin{enumerate}
\item[(A1)] $R_0=\{(x,x):x\in X\}$;
\item[(A2)] for each $i$, there exists $j$ such that the inverse of $R_i$ equals $R_j$;
\item[(A3)] if $(x,y)\in R_k$, then the number of $z\in X$ such that $(x,z)\in R_i$ and $(z,y)\in R_j$ is a constant $p^k_{ij}$ depending only on $i$, $j$, and $k$, but not on the particular choice of $x$ and~$y$.
\end{enumerate}
If the inverse of $R_i$ equals $R_i$ for all $i$, then the association scheme is called \emph{symmetric}, and if $p^k_{ij}=p^k_{ji}$ for all $i$, $j$, and $k$, then it is called \emph{commutative}. Note that symmetry implies commutativity.
\par
Let $(X,(R_i))$ be a commutative association scheme with $n$ classes. Let $A_i$ be the adjacency matrix of the digraph $(X,R_i)$. The vector space generated by $A_0,A_1,\dots,A_n$ over the complex numbers has dimension $n+1$ and is called the \emph{Bose-Mesner algebra} of the association scheme. There exists another uniquely defined basis for this vector space, consisting of minimal idempotent matrices $E_0,E_1,\dots,E_n$. We may write
\[
A_i=\sum_{k=0}^nP_i(k)E_k\quad\text{and}\quad E_k=\frac{1}{\abs{X}}\sum_{i=0}^nQ_k(i)A_i
\]
for some uniquely determined numbers $P_i(k)$ and $Q_k(i)$, called the \emph{$P$-numbers} and the \emph{$Q$-numbers} of $(X,(R_i))$, respectively. If there is an ordering of $E_0,E_1,\dots,E_n$ such that $P_k(i)=Q_i(k)$ for all $i$ and $k$, then the association scheme is called \emph{formally self-dual}. 
\par
Next we recall the definition of the association scheme of symmetric bilinear forms. Let $q$ be an odd prime power and let $\F_q$ be the finite field with $q$ elements. Let $V=V(m,q)$ be a vector space over $\F_q$ of dimension $m$ and let $X=X(m,q)$ be the set of symmetric bilinear forms defined on $V$. Note that $X$ itself is a vector space over $\F_q$ of dimension $m(m+1)/2$.
\par
Let $G=\GL(V)\rtimes X$ be the semidirect product of $\GL(V)$ and $X$. Then $G$ is a permutation group that acts transitively on $X$ as follows
\begin{equation}
\begin{split}
G\times X&\to X\\
((T,A),B)&\mapsto B'+A,
\end{split}
\label{eqn:group_action}
\end{equation}
where $B'(x,y)=B(Tx,Ty)$. The action of $G$ extends to $X\times X$ componentwise and so partitions $X\times X$ into orbits, which define the relations of a commutative association scheme~\cite{HuoWan1993}. It is well known~\cite{WanWanMaMa2003} that this association scheme has $2m$ classes and is formally self-dual (in fact it is self-dual in the strong sense of~\cite[Definition~11]{DelLev1998}). It is also well known~\cite{HuoWan1993} (and indeed easily verified) that, if $q\not\equiv 3\pmod 4$, then this association scheme is symmetric (and hence commutative), otherwise it is not symmetric, yet commutative.\footnote{When $q$ is a power of two, a symmetric association scheme with $m+\floor{m/2}$ classes arises in the same way~\cite{WanMa1996}. Except if $(m,q)=(2,2)$, this association scheme is not formally self-dual~\cite{WanWanMaMa2003}.}
\par
We are going to describe the relations explicitly, following~\cite{HuoWan1993}. Let $\alpha_1,\dots,\alpha_m$ be a basis for $V$ over $\F_q$. It will often be convenient to identify a symmetric bilinear form $B\in X$ with the $m\times m$ symmetric matrix
\begin{equation}
(B_{ij}=B(\alpha_i,\alpha_j))_{1\le i,j\le m}.   \label{eqn:bilinear_to_matrix}
\end{equation}
Since $B$ is bilinear, this matrix describes the mapping $B$ completely. The rank of this matrix is independent of the choice of the basis and is defined to be the \emph{rank} of~$B$. It is well known (see~\cite[Theorem~11.5]{Wan2003}, for example) that there exists a basis for $V$ over $\F_q$ such that the matrix of $B$ is a diagonal matrix, whose diagonal is $[z,1,\dots,1,0,\dots,0]$ for some $z\in\F_q$. Let $\eta$ be the quadratic character of $\F_q$. Then $\eta(z)$ is well defined and called the \emph{type} of $B$ (an empty product equals $1$ by convention and so the all-zero matrix has type $1$).
\par
Let $X_{i,\t}$ be the set of all elements of $X$ that have rank $i$ and type $\t$. Then, after a suitable ordering of the relations,
\begin{equation}
R_{i,\t}=\{(A,B)\in X\times X:A-B\in X_{i,\t}\}    \label{eqn:R_X}
\end{equation}
are the orbits under the action of $G$ on $X\times X$, and so $(X,(R_{i,\t}))$ is the association scheme of symmetric bilinear forms on $V$. For the remainder of this paper we assume this natural ordering of the relations. By~\eqref{eqn:R_X}, the relation containing $(A,B)$ depends only on $A-B$, which is the defining property of a translation scheme.
\par
The numbers
\[
v(i,\t)=\abs{X_{i,\t}}
\]
are called the \emph{valencies} of $(X,(R_{i,\t}))$, which have been already computed in~\cite{Car1954} (see also~\cite[Proposition~2]{HuoWan1993}). In what follows we continue to use $\eta$ for the quadratic character of $\F_q$.
\begin{proposition}[{\cite[Theorem~3]{Car1954}}]
\label{pro:num_sym_mat}
We have
\begin{align*}
v(2s,\t)&=\frac{(q^s+\eta(-1)^s\t)}{2}\,\frac{(q^m-1)(q^m-q)\cdots(q^m-q^{2s-1})}{(q^{2s}-1)(q^{2s}-q^2)\cdots(q^{2s}-q^{2s-2})},\\[1ex]
v(2s+1,\t)&=\frac{1}{2q^s}\,\frac{(q^m-1)(q^m-q)\cdots(q^m-q^{2s})}{(q^{2s}-1)(q^{2s}-q^2)\cdots(q^{2s}-q^{2s-2})}.
\end{align*}
\end{proposition}
\par
In the remainder of this section we establish the $P$-numbers and the $Q$-numbers of $(X,(R_{i,\t}))$. Since $(X,(R_{i,\t}))$ is a translation scheme, these numbers can be expressed in terms of characters of $(X,+)$ (see~\cite[Section~V]{DelLev1998} or~\cite[Section 6]{MarTan2009} for a detailed treatment). We shall identify  symmetric bilinear forms with their matrices~\eqref{eqn:bilinear_to_matrix}. 
\par
Let $\chi:\F_q\to\mathbb{C}$ be a nontrivial character of $(\F_q,+)$ and, for $A,B\in X$, write
\begin{equation}
\langle A,B\rangle=\chi(\tr(AB)),   \label{eqn:def_inner_product}
\end{equation}
where $\tr$ is the matrix trace. For all $A,A',B\in X$, we have
\begin{equation}
\langle A+A',B\rangle=\langle A,B\rangle\langle A',B\rangle.   \label{eqn:ip_homomorphism}
\end{equation}
Indeed, it is readily verified that the mapping $\langle \,\cdot\,,B\rangle$ ranges through all characters of $(X,+)$ as $B$ ranges over $X$. It can be shown~\cite[Proposition~3.2]{WanWanMaMa2003} that the numbers
\begin{equation}
P_{i,\t}(k,\e)=\sum_{A\in X_{i,\t}}\overline{\langle A,B\rangle}\quad\text{for $B\in X_{k,\e}$}   \label{eqn:P_numbers}
% Q_{k,\e}(i,\t)&=\sum_{B\in X_{k,\e}}\phi_B(A) && \quad\text{for $A\in X_{i,\t}$}
\end{equation}
are well defined. Indeed $P_{i,\t}(k,\e)$ are the $P$-numbers and
\begin{equation}
Q_{k,\e}(i,\t)=\overline{P_{k,\e}(i,\t)}   \label{eqn:Q_from_P}
\end{equation}
are the $Q$-numbers of $(X,(R_{i,\t}))$  (see~\cite{DelLev1998} or~\cite{MarTan2009}, for example). In order to give explicit expressions for the $Q$-numbers (and therefore the $P$-numbers), we need \emph{$q^2$-analogs of binomial coefficients}, which are defined by
\[
{n\brack k}=\prod_{i=1}^k(q^{2n-2i+2}-1)/(q^{2i}-1)
\]
for integral $n$ and $k$ with $k\ge 0$. We also need the following numbers, which can be derived from generalised Krawtchouk polynomials~\cite{DelGoe1975},~\cite{Del1976}. We define 
\[
F^{(m)}_r(s)=\sum_{j=0}^r(-1)^{r-j}q^{(r-j)(r-j-1)}{n-j\brack n-r}{n-s\brack j}\, c^j,
\]
where
\[
n=\left\lfloor m/2\right\rfloor,\quad\text{and}\quad c=q^{m(m-1)/(2n)},
\]
whenever this expression is defined and let $F^{(m)}_r(s)=0$ otherwise. Equivalently, these numbers can be defined via the $n+1$ equations 
\begin{equation}
\sum_{r=0}^j{n-r\brack n-j} F^m_r(s)={n-s\brack j}c^j\quad\text{for $j\in\{0,1,\dots,n\}$}   \label{eqn:ev_transform}
\end{equation}
(see~\cite[(29)]{DelGoe1975}). Let
\begin{equation}
\gamma_q=\sum_{y\in\F_q^*}\eta(y)\chi(y)   \label{eqn:def_gauss_sum}
\end{equation}
be a quadratic Gauss sum, which can be explicitly evaluated\footnote{The evaluation depends on the choice of $\chi$ and so does~\eqref{eqn:def_inner_product}. According to Theorem~\ref{thm:eigenvalues}, different choices for $\chi$ can swap the roles of $Q_{2r+1,+1}(2s+1,\t)$ and $Q_{2r+1,-1}(2s+1,\t)$ and so can change the order of the idempotent basis for the Bose-Mesner algebra.} (see~\cite[Theorem~5.12]{LidNie1997}, for example).
\par
We are now ready to give explicit expressions for the $Q$-numbers of $X$.
\begin{theorem}
\label{thm:eigenvalues}
The $Q$-numbers of $X(m,q)$ are as follows. We have $Q_{0,1}(i,\t)=1$ and $Q_{k,\e}(0,1)=v(k,\e)$, and for $k,i\ge 1$, the numbers $Q_{k,\e}(i,\t)$ are given by
\begin{align*}
2Q_{2r+1,\e}(2s+1,\t)&=-q^{2r}F^{(m-1)}_r(s)+\e\t\, \eta(-1)^{s+r}\,q^{m-s+r-1}\,\gamma_q\,F^{(m-1)}_r(s),\\[1.5ex]
2Q_{2r,\e}(2s+1,\t)&=q^{2r}F^{(m-1)}_r(s)+\e\,\eta(-1)^rq^rF^{(m)}_r(s),\\[1.5ex]
2Q_{2r+1,\e}(2s,\t)&=-q^{2r}F^{(m-1)}_r(s-1)+\t\,\eta(-1)^sq^{m-s+2r}F^{(m-2)}_r(s-1),\\[1.5ex]
2Q_{2r,\e}(2s,\t)&=q^{2r}F^{(m-1)}_r(s-1)-\t\,\eta(-1)^sq^{m-s+2r-2}F^{(m-2)}_{r-1}(s-1)\\[.5ex]
&\qquad+\e\,\eta(-1)^rq^rF^{(m)}_r(s).
\end{align*}
\end{theorem}
\begin{proof}
See appendix.
\end{proof}
\par
We close this section by recalling some useful identities for $q^2$-analogs of binomial coefficients that will be repeatedly used in the remainder of this paper (see~\cite{DelGoe1975} and~\cite{And1976}, for example):
\begin{equation}
\sum_{j=0}^hq^{j(j-1)}{h\brack j}x^{h-j}\,y^j=\prod_{k=0}^{h-1}(x+q^{2k}y)\quad\text{for real $x,y$},   \label{eqn:q-binomial}
\end{equation}
\begin{equation}
{n\brack k}{k\brack i}={n\brack i}{n-i\brack n-k},   \label{eqn:cross_identity}
\end{equation}
\begin{equation}
{n\brack k}=q^{2k}{n-1\brack k}+{n-1\brack k-1}={n-1\brack k}+q^{2(n-k)}{n-1\brack k-1}.   \label{eqn:Pascal_triangle}
\end{equation}

%%%%%%%%%%%%%%%%%%%%%%%%%%%%%%%%%%%%%%%%%%%%%%%%%%%%%%%%%%%%%%%

\section{Combinatorial properties of subsets of $X$}
\label{sec:properties}

\subsection{Definitions and basic properties}

Let $Y$ be a subset of $X=X(m,q)$ and associate with $Y$ the rational numbers 
\[
a_{i,\t}=\frac{\abs{(Y\times Y)\cap R_{i,\t}}}{\abs{Y}}.
\]
Stated differently, $a_{i,\t}$ is the average number of pairs in $Y\times Y$ whose difference has rank $i$ and type~$\t$. The sequence of numbers $(a_{i,\t})$ is called the \emph{inner distribution} of~$Y$. Let $Q_{k,\e}(i,\t)$ be the $Q$-numbers of $(X,(R_{i,\t}))$. The \emph{dual inner distribution} of~$Y$ is the sequence of numbers $(a'_{k,\e})$, where
\begin{equation}
a'_{k,\e}=\sum_{i,\t}\,Q_{k,\e}(i,\t)\,a_{i,\t}.   \label{eqn:def_dual_distribution}
\end{equation}
\par
It is readily verified that the mapping $\rho:X\times X\to\Z$, given by
\[
\rho(A,B)=\rank(A-B),
\]
is a distance function on $X$. Accordingly, given an integer $d$ satisfying $1\le d\le m$, we say that $Y$ is a \emph{$d$-code} in $X$ if $\rank(A-B)\ge d$ for all distinct $A,B\in Y$. Alternatively, $Y$ is a $d$-code if
\[
a_{i,1}=a_{i,-1}=0 \quad\text{for each $i\in\{1,2,\dots,d-1\}$}.
\]
We say that $Y$ is a \emph{$t$-design} if
\[
a'_{k,1}=a'_{k,-1}=0\quad\text{for each $k\in\{1,2,\dots,t\}$}.
\]
We also say that $Y$ is a $(2t+1,\e)$-design if it is a $(2t+1)$-design and $a'_{2t+2,\e}=0$.\footnote{Since $Q_{2t+2,\e}(i,\t)$ is independent of the choice of the character $\chi$ by Theorem~\ref{thm:eigenvalues}, in view of~\eqref{eqn:def_dual_distribution}, a $(2t+1,\e)$-design is well defined.}
\par
A subset $Y$ of $X$ is \emph{additive} if $Y$ is a subgroup of $(X,+)$. Note that the inner distribution $(a_{i,\t})$ of an additive subset $Y$ of $X$ satisfies
\[
a_{i,\t}=\abs{Y\cap X_{i,\t}},
\]
where $X_{i,\t}$ is the set of matrices in $X$ of rank $i$ and type $\t$. We define the \emph{dual} of an additive subset $Y$ of $X$ to be
\begin{equation}
Y^\perp=\{B\in X:\langle A,B\rangle=1\;\text{for each $A\in Y$}\},   \label{eqn:def_dual}
\end{equation}
which again is an additive subset of $X$. We have
\[
\abs{Y}\,\abs{Y^\perp}=\abs{X}.
\]
The following result is a special case of a general property of association schemes (see~\cite[Theorem~27]{DelLev1998}, for example).
\begin{theorem}
\label{thm:dual}
Let $Y$ be an additive subset of $X(m,q)$ and let $Y^\perp$ be its dual, having inner distributions $(a_{i,\d})$ and $(a^\perp_{k,\e})$, respectively. Then $(a^\perp_{k,\e})$ is proportional to the dual inner distribution $(a'_{k,\e})$ of $Y$. In particular,
\[
\abs{Y}a^\perp_{k,\e}=a'_{k,\e}.
\]
\end{theorem}
\par
In the remainder of this section we give an equivalent formulation of~\eqref{eqn:def_dual_distribution}, which will make our forthcoming analysis easier. To do so, write
\begin{equation}
\begin{split}
A_s&=a_{2s,1}+a_{2s,-1}+a_{2s-1,1}+a_{2s-1,-1},\\[1ex]
B_s&=a_{2s,1}+a_{2s,-1}+a_{2s+1,1}+a_{2s+1,-1},\\[1ex]
C_s&=\eta(-1)^sq^{-s}(a_{2s,1}-a_{2s,-1}),\\[1ex]
D_s&=\eta(-1)^sq^{-s}(a_{2s+1,1}-a_{2s+1,-1}),
\end{split}   \label{eqn:def_ABCD}
\end{equation}
and
\begin{equation}
\begin{split}
A'_r&=a'_{2r,1}+a'_{2r,-1}+a'_{2r-1,1}+a'_{2r-1,-1},\\[1ex]
B'_r&=a'_{2r,1}+a'_{2r,-1}+a'_{2r+1,1}+a'_{2r+1,-1},\\[1ex]
C'_r&=\eta(-1)^rq^{-r}(a'_{2r,1}-a'_{2r,-1}),\\[1ex]
D'_r&=\eta(-1)^rq^{-r}(a'_{2r+1,1}-a'_{2r+1,-1}).
\end{split}   \label{eqn:def_ABCD_dual}
\end{equation}
It is immediate that the knowledge of $(A_s)$, $(B_s)$, $(C_s)$, and $(D_s)$ is equivalent to the knowledge of $(a_{i,\t})$. The reason for introducing these numbers becomes clear from the following lemma.
\begin{lemma}
\label{lem:four_equations}
Consider a subset of $X(m,q)$ with inner distribution $(a_{i,\t})$ and dual inner distribution $(a'_{k,\e})$ and let $A_s,B_s,C_s,D_s$ and $A_r',B_r',C_r',D_r'$ be as defined in~\eqref{eqn:def_ABCD} and~\eqref{eqn:def_ABCD_dual}, respectively. Then
\begin{align*}
A_r'&=\sum_sF^{(m+1)}_r(s)\,A_s,\\
C_r'&=\sum_sF^{(m)}_r(s)\,B_s,\\
B_r'&=q^m\sum_sF^{(m)}_r(s)\,C_s,\\
D_r'&=q^{m-1}\gamma_q\sum_sF^{(m-1)}_r(s)\,D_s,
\end{align*}
where $\gamma_q$ is the Gauss sum defined in~\eqref{eqn:def_gauss_sum}.
\end{lemma}
\begin{proof}
From Theorem~\ref{thm:eigenvalues} we find that, for each $i\in\{2s,2s-1\}$, we have
\begin{align*}
\sum_{\e\in\{-1,1\}}\big[Q_{2r,\e}(i,\t)+Q_{2r-1,\e}(i,\t)\big]&=q^{2r}F^{(m-1)}_r(s-1)-q^{2r-2}F^{(m-1)}_{r-1}(s-1)\\
&=F^{(m+1)}_r(s),
\end{align*}
using~\eqref{eqn:Pascal_triangle}. In view of~\eqref{eqn:def_dual_distribution}, this establishes the expression for $A'_r$. We proceed similarly for the remaining three equations. The expression for $B'_r$ follows from
\begin{gather*}
\sum_{\e\in\{-1,1\}}\big[Q_{2r,\e}(2s,\t)+Q_{2r+1,\e}(2s,\t)\big]=\t\,\eta(-1)^sq^{m-s}F^{(m)}_r(s),
\intertext{and}
\sum_{\e\in\{-1,1\}}\big[Q_{2r,\e}(2s+1,\t)+Q_{2r+1,\e}(2s+1,\t)\big]=0.
\end{gather*}
The expression for $C'_r$ follows from
\[
\sum_{\e\in\{-1,1\}}\e\,Q_{2r,\e}(i,\t)=\eta(-1)^rq^r\,F^{(m)}_r(s)
\]
for each $i\in\{2s,2s+1\}$, and the expression for $D'_r$ follows from
\begin{gather*}
\sum_{\e\in\{-1,1\}}\e\, Q_{2r+1,\e}(2s+1,\t)=\t\,\eta(-1)^{s+r}\,q^{m-s+r-1}\,\gamma_q\,F^{(m-1)}_r(s),
\intertext{and}
\sum_{\e\in\{-1,1\}}\e\, Q_{2r+1,\e}(2s,\t)=0.   \qedhere
\end{gather*}
\end{proof}

%%%%%%%%%%%%%%%%%%%%%%%%%%%%%%%%%%%%%%%%%%%%%%%%%%%%%%

\subsection{Bounds on codes}

In this section we prove bounds on the size of $d$-codes in $X(m,q)$. Our main result is the following bound for additive codes.
\begin{theorem}
\label{thm:bound_sg}
Let $Y$ be an additive $d$-code in $X(m,q)$. Then we have
\[
\abs{Y}\le
\begin{cases}
q^{m(m-d+2)/2}     & \text{for $m-d$ even}\\[1ex]
q^{(m+1)(m-d+1)/2} & \text{for $m-d$ odd}.
\end{cases}
\]
\end{theorem}
\par
We shall see in Section~\ref{sec:constructions} that the bound in Theorem~\ref{thm:bound_sg} is best possible. Theorem~\ref{thm:bound_sg} will follow from the slightly stronger Lemma~\ref{lem:bound_odd_d} and Lemma~\ref{lem:bound_even_d_sg}, to be stated and proved below.
\par
We shall need the following result, which is a consequence of a well known property of association schemes (see~\cite[Theorem~3]{DelLev1998}, for example).
\begin{lemma}
\label{lem:dual_nonnegative}
Let $Y$ be a subset of $X(m,q)$ with dual inner distribution $(a'_{k,\e})$. Then the numbers $a'_{k,\e}$ are real and nonnegative.
\end{lemma}
\par
We now consider $d$-codes in $X(m,q)$ for odd $d$ that are not necessarily additive.
\begin{lemma}
\label{lem:bound_odd_d}
Let $Y$ be a $(2\d-1)$-code in $X(m,q)$. Then
\[
\abs{Y}\le \begin{cases}
q^{m((m+1)/2-\d+1)} & \text{for odd $m$}\\[1ex]
q^{(m+1)(m/2-\d+1)} & \text{for even $m$}.
\end{cases}
\]
Moreover, equality occurs if and only if $Y\!$ is a $(2t+2)$-design, where $t=\floor{(m+1)/2}-\d$.
\end{lemma}
\begin{proof}
Let $(a_{i,\t})$ be the inner distribution of $Y$ and let $A_s$ and $A'_r$ be as defined in~\eqref{eqn:def_ABCD} and~\eqref{eqn:def_ABCD_dual}, respectively. Write 
\[
n=\floor{(m+1)/2}\quad\text{and}\quad c=q^{m(m+1)/(2n)}.
\]
From Lemma~\ref{lem:four_equations} and~\eqref{eqn:ev_transform} we find that
\[
\sum_{r=0}^{n-\d+1}{n-r\brack \d-1} A'_r=c^{n-\d+1}\sum_{s=0}^n{n-s\brack n-\d+1}A_s.
\]
Since $Y$ is a $(2\d-1)$-code, we have $A_s=0$ for $0<s<\d$. Since $A_0=1$ and $A'_0=\abs{Y}$, we obtain
\begin{equation}
\sum_{r=1}^{n-\d+1}{n-r\brack \d-1} A'_r={n\brack \d-1}(c^{n-\d+1}-\abs{Y}).   \label{eqn:moments_A}
\end{equation}
Since the left-hand side must be real and nonnegative by Lemma~\ref{lem:dual_nonnegative}, we deduce 
\[
\abs{Y}\le c^{n-\d+1},
\]
as required. Moreover, in case of equality, the left-hand side of~\eqref{eqn:moments_A} equals zero, which occurs if and only if $Y$ is a $(2n-2\d+2)$-design.
\end{proof}
\par
We shall see in Theorem~\ref{thm:inner_dist_odd} that, in case of equality in Lemma~\ref{lem:bound_odd_d}, the inner distribution of~$Y$ is uniquely determined. It should be noted that the statement of Lemma~\ref{lem:bound_odd_d} was already pointed out in~\cite[Section~5]{Sch2010} and is also known to be true in characteristic two~\cite[Corollary~7]{Sch2010}.
\par
For even $d$, we have the following bound for the size of $d$-codes in $X(m,q)$, which however applies only to additive sets.
\begin{lemma}
\label{lem:bound_even_d_sg}
Let $Y$ be an additive $(2\d)$-code in $X(m,q)$. Then we have
\[
\abs{Y}\le \begin{cases}
q^{m(m/2-\d+1)}         & \text{for even $m$}\\[1ex]
q^{(m+1)((m-1)/2-\d+1)} & \text{for odd $m$}.
\end{cases}
\]
\end{lemma}
\begin{proof}
Let $(a_{i,\t})$ be the inner distribution of $Y$ and let $C_s$ and $B'_r$ be as defined in~\eqref{eqn:def_ABCD} and~\eqref{eqn:def_ABCD_dual}, respectively. Write 
\[
n=\floor{m/2}\quad\text{and}\quad c=q^{m(m-1)/(2n)}.
\]
From Lemma~\ref{lem:four_equations} and~\eqref{eqn:ev_transform} we find that
\[
q^{n-\d+1}\sum_{r=0}^{n-\d+1}{n-r\brack \d-1} C'_r=(cq)^{n-\d+1}\sum_{s=0}^n{n-s\brack n-\d+1}B_s.
\]
Since $Y$ is a $(2\d)$-code, only the first summand on the right-hand side is nonzero and, using $B_0=1$, we then have
\[
q^{n-\d+1}\sum_{r=0}^{n-\d+1}{n-r\brack \d-1} C'_r=(cq)^{n-\d+1}{n\brack \d-1}.
\]
Since $Y$ is additive, we find from Theorem~\ref{thm:dual} and the definition of the numbers $C'_r$ that the left-hand side is divisible by $\abs{Y}$. Hence the right-hand side is divisible by $\abs{Y}$. Let $p$ be the prime dividing $q$. Note that $\abs{Y}$ must be a power of $p$ by Lagrange's theorem and that ${n\brack \d-1}$ is not divisible by $p$. Therefore $\abs{Y}$ divides $(cq)^{n-\d+1}$, which gives
\[
\abs{Y}\le(cq)^{n-\d+1},
\]
as required
\end{proof}
\par
We note that, if $m$ is even and $Y$ is an $m$-code in $X(m,q)$ of size $q^m$, then $Y$ must be a $(1,\eta(-1))$-design. This can be proved by considering the number $A'_1+qC'_1$, which is associated with $Y$ in the usual way. We shall see in Proposition~\ref{pro:inner_dist_even} that in this case the inner distribution of $Y$ is uniquely defined (to be precise, the two types of forms of rank $m$ occur equally often). However, this is not true in general: In case of equality in Lemma~\ref{lem:bound_even_d_sg}, the inner distribution of $Y$ is not always uniquely defined. This is essentially because there is no analogue of the second statement in Lemma~\ref{lem:bound_odd_d}. For example, consider $2$-codes that are subspaces of $X(4,3)$. According to Lemma~\ref{lem:bound_even_d_sg}, the largest dimension of such a subspace is $8$ and in case of equality it can be verified with a computer that exactly four different inner distributions occur. These are given in Table~\ref{tab:id_X43}.
\begin{table}[ht]
\centering
\caption{The possible inner distributions of $2$-codes that are $8$-dimensional subspaces of $X(4,3)$.}
\label{tab:id_X43}
\begin{tabular}{ccccccccc}
\hline
$a_{1,0}$ & $a_{1,1}$ & $a_{1,-1}$ & $a_{2,1}$ & $a_{2,-1}$ & $a_{3,1}$ & $a_{3,-1}$ & $a_{4,1}$ & $a_{4,-1}$ \\\hline\hline 
1 & 0 & 0 & 100 & 160 & 1080 & 1080 & 2340 & 1800\\
1 & 0 & 0 & 100 & 214 &  972 &  972 & 2340 & 1962\\
1 & 0 & 0 & 118 & 196 &  972 &  972 & 2394 & 1908\\
1 & 0 & 0 & 136 & 232 &  864 &  864 & 2448 & 2016\\\hline
\end{tabular}
\end{table}
\par
We close this section by addressing $(2\d)$-codes that are not necessarily additive. Of course we can use Lemma~\ref{lem:bound_odd_d} to obtain bounds for such codes. We now show that these bounds can be improved by a factor of roughly $q+1$. 
\begin{proposition}
\label{pro:bound_even}
Let $Y$ be a $(2\d)$-code in $X(m,q)$. Then we have
\[
\abs{Y}\le \begin{cases}
q^{m((m+1)/2-\d+1)}\;\dfrac{1+q^{-m+1}}{q+1} & \text{for odd $m$}\\[3ex]
q^{(m+1)(m/2-\d+1)}\;\dfrac{1+q^{-m+2\d-1}}{q+1} & \text{for even $m$}.
\end{cases}
\]
\end{proposition}
\begin{proof}
Let $(a_{i,\t})$ be the inner distribution of $Y$ and let $A_s$, $B_s$, $C_s$ be as defined~\eqref{eqn:def_ABCD}. Let $(a'_{k,\e})$ be the dual inner distribution of $Y$ and let  $A'_r$, $B'_r$, $C'_r$ be as defined~\eqref{eqn:def_ABCD_dual}.
\par
First let $m$ be even, say $m=2n$. From Lemma~\ref{lem:four_equations} and~\eqref{eqn:ev_transform} we find that
\[
\sum_{r=0}^{n-\d+1}{n-r\brack \d-1} \big(A'_r+q C'_r\big)=q^{(m-1)(n-\d+1)}\sum_{s=0}^n{n-s\brack n-\d+1}\big(q^{2(n-\d+1)}A_s+qB_s\big).
\]
Since $Y$ is a $(2\d)$-code, only the first summand on the right-hand side is nonzero. Since $A_0=B_0=1$ and $A'_0=C'_0=\abs{Y}$, we then have
\[
\sum_{r=1}^{n-\d+1}{n-r\brack \d-1} \big(A'_r+q\, C'_r\big)={n\brack \d-1} \big(q^{(m-1)(n-\d+1)}(q^{2(n-\d+1)}+q\big)-(q+1)\abs{Y}\big).
\]
The left-hand side must be real and nonnegative by Lemma~\ref{lem:dual_nonnegative}, which gives the bound for $\abs{Y}$.
\par
Now let $m$ be odd, say $m=2n+1$. In this case, we find similarly as before that
\[
\sum_{r=0}^{n-\d+1}{n-r\brack \d-1} \big(B'_r+q C'_r\big)=q^{m(n-\d+1)}\sum_{s=0}^n{n-s\brack n-\d+1}\big(q^mC_s+qB_s\big).
\]
Again, only the first summand on the right-hand side is nonzero. Since $B_0=C_0=1$ and $B'_0=\abs{Y}+a'_{1,1}+a'_{1,-1}$ and $C'_0=\abs{Y}$ we then have
\begin{multline*}
{n\brack \d-1}(a'_{1,1}+a'_{1,-1})+\sum_{r=1}^{n-\d+1}{n-r\brack \d-1} \big(B'_r+qC'_r\big)\\
={n\brack \d-1}\big(q^{m(n-\d+1)}(q^m+q)-(q+1)\abs{Y}\big).
\end{multline*}
Since the left-hand side must be real and nonnegative, we obtain the bound for $\abs{Y}$. 
\end{proof}
\par
Consider a subset $Y$ with inner distribution $(a_{i,\t})$ and dual inner distribution $(a'_{k,\e})$. There is evidence that, for even $m$, Proposition~\ref{pro:bound_even} gives the optimal solution to the linear program whose objective is to maximise
\[
\abs{Y}=\sum_{i,\t}a_{i,\t},
\]
subject to the nonnegativity of the numbers $a_{i,\t}$ and $a'_{k,\e}$ (which is forced by Lemma \ref{lem:dual_nonnegative}). We have checked this using the simplex algorithm for several small values of $m$ and~$q$.
\par
We emphasise that $(2\d)$-codes in $X(m,q)$ are still quite mysterious since it is not clear whether Lemma~\ref{lem:bound_even_d_sg} also holds for general $(2\d)$-codes or whether there exist nonadditive $(2\d)$-codes that are better than additive codes.

%%%%%%%%%%%%%%%%%%%%%%%%%%%%%%%%%%%%%%%%%%%%%%%%%%%%%%%%%%%%%%%

\subsection{Determination of inner distributions}

In this section we determine the inner distributions of $d$-codes in $X(m,q)$, provided that their dual inner distributions contain enough zero entries. The method uses Lemma~\ref{lem:four_equations} and the following lemma, which is implicit in the proof of~\cite[Theorem~4]{DelGoe1975}.
\begin{lemma}
\label{lem:eq_sys_solution}
Let $m$ and $\d$ be positive integers and write
\[
n=\lfloor m/2\rfloor\quad\text{and}\quad c=q^{m(m-1)/(2n)}.
\]
Let $w_0,\dots,w_n$ be real numbers with $w_1=\dots=w_{\d-1}=0$ and write
\[
w'_r=\sum_{s=0}^nF^{(m)}_r(s)\,w_s.
\]
Then
\[
w_i=\sum_{j=0}^{i-\d}(-1)^jq^{j(j-1)}\bigg({n\brack i}{i\brack j}\bigg(\frac{w'_0}{c^{n+j-i}}-w_0\bigg)+{n+j-i\brack n-i}\sum_{r=1}^{n+j-i}{n-r\brack i-j}\frac{w'_r}{c^{n+j-i}}\bigg)
\]
for $i>0$. In particular, if $w'_1=\dots=w'_{n-\d}=0$, then
\[
w_i={n\brack i}\sum_{j=0}^{i-\d}(-1)^jq^{j(j-1)}{i\brack j}\bigg(\frac{w'_0}{c^{n+j-i}}-w_0\bigg)
\]
for $i>0$.
\end{lemma}
\begin{proof}
Multiply both sides of~\eqref{eqn:ev_transform} by $w_s$ and then sum over $s$ to obtain
\[
\sum_{r=0}^j{n-r\brack n-j}w'_r=c^j \sum_{s=0}^n {n-s\brack j} w_s
\]
for each $j\in\{0,1,\dots,n\}$. Since $w_1=\dots=w_{\d-1}=0$, we obtain
\[
\frac{1}{c^j}\sum_{r=0}^j{n-r\brack n-j}w'_r-{n\brack j}w_0=\sum_{i=0}^{n-\d} {i\brack j}w_{n-i}
\]
for each $j\in\{0,1,\dots,n\}$. Then, by the inversion formula
\[
\sum_{j=i}^k(-1)^{j-i}q^{(j-i)(j-i-1)}{j\brack i}{k\brack j}=
\begin{cases}
1 & \text{for $k=i$}\\
0 & \text{otherwise}
\end{cases}
\]
(which can be derived from~\eqref{eqn:q-binomial} and~\eqref{eqn:cross_identity}), we have
\[
w_{n-i}=\sum_{j=i}^{n-\d}(-1)^{j-i}q^{(j-i)(j-i-1)}{j\brack i}\bigg(\frac{1}{c^j}\sum_{r=0}^j{n-r\brack n-j}w'_r-{n\brack j}w_0\bigg),
\]
from which we can obtain the desired expression for $w_i$ using elementary manipulations including~\eqref{eqn:cross_identity}.
\end{proof}
\par
Our next result gives the inner distribution of $d$-codes for odd $d$ and in particular applies to $(2\d-1)$-codes meeting the bounds in Lemma~\ref{lem:bound_odd_d}.
\begin{theorem}
\label{thm:inner_dist_odd}
If $Y$ is a $(2\d-1)$-code and a $(2n-2\d+3)$-design in $X(2n+1,q)$, then its inner distribution $(a_{i,\t})$ satisfies
\begin{gather*}
a_{2i-1,\t}=\frac{1}{2}{n\brack i-1}\sum_{j=0}^{i-\d}(-1)^{j}q^{j(j-1)}{i\brack j}\bigg(\frac{\abs{Y}}{q^{(2n+1)(n+1+j-i)}}-1\bigg),\\[1ex]
a_{2i,\t}=\frac{1}{2}\big(q^{2i}+\t\eta(-1)^iq^i\big)\,{n\brack i}\sum_{j=0}^{i-\d}(-1)^jq^{j(j-1)}{i\brack j}\bigg(\frac{\abs{Y}}{q^{(2n+1)(n+1+j-i)}}-1\bigg)
\end{gather*}
for $i>0$. If $Y$ is a $(2\d-1)$-code and a $(2n-2\d+2)$-design in $X(2n,q)$, then its inner distribution $(a_{i,\t})$ satisfies
\begin{align*}
a_{2i-1,\t}&=\frac{1}{2}(q^{2i}-1){n\brack i}\sum_{j=0}^{i-\d}(-1)^{j}q^{j(j-1)}{i-1\brack j}\frac{\abs{Y}\,q^{2j}}{q^{(2n+1)(n+1+j-i)}},\\[1ex]
a_{2i,\t}&=\frac{1}{2}{n\brack  i}\sum_{j=0}^{i-\d+1}(-1)^{j}q^{j(j-1)}{i\brack j}\bigg(\frac{\abs{Y}\,q^{2j}}{q^{(2n+1)(n+j-i)}}-1\bigg)\\
&+\frac{\t}{2}\,\eta(-1)^iq^i{n\brack i}\sum_{j=0}^{i-\d}(-1)^jq^{j(j-1)}{i\brack j}\bigg(\frac{\abs{Y}}{q^{(2n-1)(n+j-i)}\,q^{2n}}-1\bigg)
\end{align*}
for $i>0$.
\end{theorem}
\begin{proof}
Given a subset $Y$ of $X(m,q)$ with inner distribution $(a_{i,\t})$, let $A_s,B_s,C_s,D_s$ be defined as in~\eqref{eqn:def_ABCD}, let $(a'_{k,\e})$ be the dual inner distribution of $Y$, and let $A'_r,B'_r,C'_r,D'_r$ be defined as in~\eqref{eqn:def_ABCD_dual}.
\par
First suppose that $Y$ is a $(2\d-1)$-code and a $(2n-2\d+3)$-design in $X(2n+1,q)$. This implies
\begin{align*}
A_1=\cdots=A_{\d-1}&=0, &  A'_1=\cdots=A'_{n-\d+1}&=0,\\
B_1=\cdots=B_{\d-2}&=0, &  C'_1=\cdots=C'_{n-\d+1}&=0,\\
C_1=\cdots=C_{\d-1}&=0, &  B'_1=\cdots=B'_{n-\d+1}&=0,\\
D_1=\cdots=D_{\d-2}&=0, &  D'_1=\cdots=D'_{n-\d+1}&=0.
\end{align*}
Moreover we have $A_0=1$ and $A'_0=C'_0=\abs{Y}$. Therefore we find from Lemmas~\ref{lem:four_equations} and~\ref{lem:eq_sys_solution} that
\[
A_i={n+1\,\brack i}\sum_{j=0}^{i-\d}(-1)^jq^{j(j-1)}{i\brack j}\bigg(\frac{\abs{Y}}{q^{(2n+1)(n+1+j-i)}}-1\bigg)
\]
and
\[
B_i={n\brack i}\sum_{j=0}^{i-\d+1}(-1)^jq^{j(j-1)}{i\brack j}\bigg(\frac{\abs{Y}}{q^{(2n+1)(n+j-i)}}-B_0\bigg)
\]
for $i>0$. If $\d>1$, then $B_0=1$, and since
\begin{equation}
\sum_{j=0}^i(-1)^jq^{j(j-1)}{i\brack j}=0   \label{eqn:sum_binomials}
\end{equation}
(which is a consequence of~\eqref{eqn:q-binomial}), we conclude that
\[
B_i={n\brack i}\sum_{j=0}^{i-\d+1}(-1)^jq^{j(j-1)}{i\brack j}\bigg(\frac{\abs{Y}}{q^{(2n+1)(n+j-i)}}-1\bigg)
\]
for $i>0$ holds for all $\d\ge 1$. It is readily verified that there is a unique solution for $a_{i,1}+a_{i,-1}$ to the equation system consisting of the first and the second equation of~\eqref{eqn:def_ABCD}. Using elementary manipulations, in particular~\eqref{eqn:Pascal_triangle}, it is readily verified that the expressions for $a_{i,1}+a_{i,-1}$ given in the theorem satisfy this system of equations.
\par
Since $C_0=1$ and $B'_0=\abs{Y}$, we find from Lemmas~\ref{lem:four_equations} and~\ref{lem:eq_sys_solution} that
\[
C_i={n\brack i}\sum_{j=0}^{i-\d}(-1)^jq^{j(j-1)}{i\brack j}\bigg(\frac{\abs{Y}}{q^{(2n+1)(n+1+j-i)}}-1\bigg)
\]
for $i>0$. We then obtain the expression for $a_{2i,\t}$ from the definition of $C_i$.
\par
Since $D'_0=0$, we find from Lemmas~\ref{lem:four_equations} and~\ref{lem:eq_sys_solution} that
\[
D_i={n\brack i}\sum_{j=0}^{i-\d+1}(-1)^jq^{j(j-1)}{i\brack j}D_0.
\]
Now either $\d>1$, in which case $D_0=0$, or $\d=1$, in which case we invoke~\eqref{eqn:sum_binomials}. In both cases we obtain $D_i=0$ for all $i$. Therefore, $a_{2i+1,1}=a_{2i+1,-1}$ for all $i$, which completes the proof of the first case.
\par
Now suppose that $Y$ is a $(2\d-1)$-code and a $(2n-2\d+2)$-design in $X(2n,q)$. This case is broadly similar to the first case, though here we have to use that $\abs{Y}=q^{(2n+1)(n-\d+1)}$, which follows from Lemma~\ref{lem:bound_odd_d}. We omit the details and restrict ourselves to provide the following expressions for the numbers $A_i$, $B_i$, $C_i$, and $D_i$:
\begin{align*}
A_i&={n\brack i}\sum_{j=0}^{i-\d}(-1)^jq^{j(j-1)}{i\brack j}\bigg(\frac{\abs{Y}}{q^{(2n+1)(n+j-i)}}-1\bigg),\\
B_i&={n\brack i}\sum_{j=0}^{i-\d+1}(-1)^jq^{j(j-1)}{i\brack j}\bigg(\frac{\abs{Y}}{q^{(2n-1)(n+j-i)}}-1\bigg),\\
C_i&={n\brack i}\sum_{j=0}^{i-\d}(-1)^jq^{j(j-1)}{i\brack j}\bigg(\frac{\abs{Y}}{q^{(2n-1)(n+j-i)}\,q^{2n}}-1\bigg)
\end{align*}
for $i>0$ and $D_i=0$ for all $i$.
\end{proof}
\par
In Section~\ref{sec:constructions} we provide unified constructions of additive $d$-codes in $X(m,q)$ that satisfy the bound of Theorem~\ref{thm:bound_sg} with equality. In the case that $d$ is odd, their inner distributions are uniquely determined by their parameters and are given by Theorem~\ref{thm:inner_dist_odd}. In the case that $d$ is even, it will turn out that their inner distributions are given by the following result, which is of a similar spirit as Theorem~\ref{thm:inner_dist_odd}. We note that, in the case that $n=2$ and $\d=1$, this result corresponds to the first inner distribution in Table~\ref{tab:id_X43}.
\begin{proposition}
\label{pro:inner_dist_even}
If $Y$ is a $(2\d)$-code and a $(2n-2\d+1)$-design in $X(2n,q)$, then its inner distribution $(a_{i,\t})$ satisfies
\begin{align*}
a_{2i-1,\t}&=\frac{1}{2}(q^{2i}-1){n\brack i}\sum_{j=0}^{i-\d-1}(-1)^jq^{j(j-1)}{i-1\brack j}\frac{\abs{Y}\,q^{2j}}{q^{(2n+1)(n+1+j-i)}},\\[1ex]
a_{2i,\t}&=\frac{1}{2}{n\brack  i}\sum_{j=0}^{i-\d}(-1)^{j}q^{j(j-1)}{i\brack j}\bigg(\frac{\abs{Y}\,q^{2j}}{q^{(2n+1)(n+j-i)}}-1\bigg)\\
&+\frac{\t}{2}\,\eta(-1)^iq^i{n\brack i}\sum_{j=0}^{i-\d}(-1)^jq^{j(j-1)}{i\brack j}\bigg(\frac{\abs{Y}}{q^{(2n-1)(n+j-i)}q^{2n}}-1\bigg)
\end{align*}
for $i>0$. If $Y$ is a $(2\d)$-code and a $(2n-2\d+1,\eta(-1)^{n-\d+1})$-design in $X(2n+1,q)$, then its inner distribution $(a_{i,\t})$ satisfies
\begin{align*}
a_{2i-1,\t}&=\frac{1}{2}{n\brack i-1}\sum_{j=0}^{i-\d}(-1)^jq^{j(j-1)}{i\brack j}\bigg(\frac{\abs{Y}}{q^{(2n+1)(n+1+j-i)}}-1\bigg)\\
&\qquad+\frac{1}{2}(-1)^{i-d}q^{(i-\d)(i-\d-1)}
{n\brack \d-1}\bigg(\frac{\abs{Y}}{q^{(2n+1)(n-\d+1)}}-1\bigg)\\[1ex]
&\qquad\qquad\times\bigg({n-\d\brack n-i+1}(q^{n-\d+1}+1)-{n-\d+1\brack n-i+1}\bigg),
\\[1ex]
a_{2i,\t}&=\frac{1}{2}\big(q^{2i}+\t\eta(-1)^iq^i\big){n\brack i}\sum_{j=0}^{i-\d}(-1)^jq^{j(j-1)}{i\brack j}\bigg(\frac{\abs{Y}}{q^{(2n+1)(n+1+j-i)}}-1\bigg)\\
&+\frac{1}{2}(-1)^{i-\d}q^{(i-\d+1)(i-d)}{n\brack \d-1}{n-\d\brack n-i}(q^{n-\d+1}+1)  \bigg(\frac{\abs{Y}}{q^{(2n+1)(n-\d+1)}}-1\bigg)
\end{align*}
for $i>0$.
\end{proposition}
\begin{proof}
Given a subset $Y$ of $X(m,q)$, we let $A_s,B_s,C_s,D_s$ and $A'_r,B'_r,C'_r,D'_r$ be as defined in the proof of Theorem~\ref{thm:inner_dist_odd}. The proof of the proposition is similar to the proof of Theorem~\ref{thm:inner_dist_odd}, and so we restrict ourselves to provide expressions for the numbers $A_i$, $B_i$, $C_i$, and $D_i$.
\par
First suppose that $Y$ is a $(2\d)$-code and a $(2n-2\d+1)$-design in $X(2n,q)$. In this case we have
\begin{align*}
A_i&={n\brack i}\sum_{j=0}^{i-\d}(-1)^jq^{j(j-1)}{i\brack j}\bigg(\frac{\abs{Y}}{q^{(2n+1)(n+j-i)}}-1\bigg),\\
B_i&={n\brack i}\sum_{j=0}^{i-\d}(-1)^jq^{j(j-1)}{i\brack j}\bigg(\frac{\abs{Y}}{q^{(2n-1)(n+j-i)}}-1\bigg),\\
C_i&={n\brack i}\sum_{j=0}^{i-\d}(-1)^jq^{j(j-1)}{i\brack j}\bigg(\frac{\abs{Y}}{q^{(2n-1)(n+j-i)}q^{2n}}-1\bigg)
\end{align*}
for $i>0$ and $D_i=0$ for all $i$.
\par
Now suppose that $Y$ is a $(2\d)$-code and a $(2n-2\d+1,\eta(-1)^{n-\d+1})$-design in $X(2n+1,q)$. From Lemmas~\ref{lem:four_equations} and~\ref{lem:eq_sys_solution} we obtain
\begin{align}
A_i&={n+1\brack i}\sum_{j=0}^{i-\d}(-1)^jq^{j(j-1)}{i\brack j}\bigg(\frac{\abs{Y}}{q^{(2n+1)(n+1+j-i)}}-1\bigg)   \label{eqn:A_nd1}\\
&\qquad+(-1)^{i-\d}q^{(i-\d)(i-\d-1)}{n-\d+1\brack n-i+1}\frac{A'_{n-\d+1}}{q^{(2n+1)(n-\d+1)}}   \nonumber
\end{align}
for $i>0$. We compute $A'_{n-\d+1}$ next. From Lemma~\ref{lem:four_equations} and~\eqref{eqn:ev_transform} we find that
\[
\sum_{r=0}^{n-\d+1}{n-r\brack \d-1}C'_r=q^{(2n+1)(n-\d+1)}\sum_{s=0}^n{n-s\brack n-\d+1}B_s,
\]
which implies
\[
{n\brack \d-1}C'_0+C'_{n-\d+1}=q^{(2n+1)(n-\d+1)}{n\brack \d-1}B_0.
\]
Hence, since $C'_0=\abs{Y}$ and $B_0=1$,
\[
C'_{n-\d+1}={n\brack \d-1}\big(q^{(2n+1)(n-\d+1)}-\abs{Y}\big).
\]
Using our assumption that $Y$ is a $(2n-2\d+1,\eta(-1)^{n-\d+1})$-design and the definitions~\eqref{eqn:def_ABCD_dual} of the numbers $A'_r$ and $C'_r$, we find that 
\[
A'_{n-\d+1}=-q^{n-\d+1}C'_{n-\d+1}.%=-q^{n-\d+1}{n\brack \d-1}\big(q^{(2n+1)(n-\d+1)}-\abs{Y}\big).
\]
Substitute into~\eqref{eqn:A_nd1} to obtain
\begin{align*}
A_i&={n+1\brack i}\sum_{j=0}^{i-\d}(-1)^jq^{j(j-1)}{i\brack j}\bigg(\frac{\abs{Y}}{q^{(2n+1)(n+1+j-i)}}-1\bigg)\\
&\qquad+(-1)^{i-\d}q^{(i-\d)(i-\d-1)}q^{n-\d+1}{n-\d+1\brack n-i+1}{n\brack \d-1}\bigg(\frac{\abs{Y}}{q^{(2n+1)(n-\d+1)}}-1\bigg)
\end{align*}
for $i>0$. Using Lemmas~\ref{lem:four_equations} and~\ref{lem:eq_sys_solution}, it is easy to find that
\begin{align*}
B_i&={n\brack i}\sum_{j=0}^{i-\d}(-1)^jq^{j(j-1)}{i\brack j}\bigg(\frac{\abs{Y}}{q^{(2n+1)(n+j-i)}}-1\bigg),\\
C_i&={n\brack i}\sum_{j=0}^{i-\d}(-1)^jq^{j(j-1)}{i\brack j}\bigg(\frac{\abs{Y}}{q^{(2n+1)(n+1+j-i)}}-1\bigg)
\end{align*}
for $i>0$ and $D_i=0$ for all $i$. As before, it is now straightforward to verify the expressions for $(a_{i,\t})$ given in the theorem.
\end{proof}

%%%%%%%%%%%%%%%%%%%%%%%%%%%%%%%%%%%%%%%%%%%%%%%%%%%%%%%%%%%%%%%%%%%%%%%%%%%%%%

\subsection{A combinatorial characterisation of designs}

Designs in association schemes often have a nice combinatorial interpretation. This has been established for many classical association schemes~\cite{Del1973},~\cite{Del1976a},~\cite{Sta1986},~\cite{Mun1986}. The main result of this section is the following combinatorial characterisation of $t$-designs in $X(m,q)$.
\begin{theorem}
\label{thm:designs}
Let $U$ be a $t$-dimensional subspace of $V(m,q)$ and let $A$ be a symmetric bilinear form on $U$. Then a subset $Y$ of $X(m,q)$ (whose forms are defined on $V(m,q)$) is a $t$-design if and only if the number of forms in $Y$ that are an extension of $A$ is independent of the choice of $U$ and $A$.
\end{theorem}
\par
In order to prove Theorem~\ref{thm:designs}, we need the notion of a \emph{$t$-antidesign} in $X(m,q)$, which we define to be a subset of $X(m,q)$ whose dual inner distribution $(a'_{k,\e})$ satisfies $a'_{k,1}=a'_{k,-1}=0$ for all $k>t$ and $a'_{k,1}a'_{k,-1}\ne 0$ for all $k$ satisfying $1\le k\le t$.
\par
Throughout the remainder of this section, let $G=\GL(V(m,q))\rtimes X(m,q)$ be the group of automorphisms that acts on $X(m,q)$ as defined by~\eqref{eqn:group_action}. The following lemma is a consequence of a result due to Roos~\cite[Theorem 3.4]{Roo1982} (see also Delsarte~\cite[Theorem~6.8]{Del1977}), which characterises $t$-designs in terms of $t$-antidesigns.\footnote{It should be noted that \cite[Theorem 3.4]{Roo1982} and~\cite[Theorem~6.8]{Del1977} were proved only for symmetric association schemes, but the proof of \cite[Theorem 3.4]{Roo1982} is easily adapted to work for commutative association schemes.}
\begin{lemma}
\label{lem:design_antidesign}
Let $Y$ and $Z$ be two subsets of $X(m,q)$ and suppose that $Z$ is a $t$-antidesign. Then $Y$ is a $t$-design if and only if $\abs{Y\cap Z}=\abs{Y\cap gZ}$ for all $g\in G$.
\end{lemma}
\par
We now prove Theorem~\ref{thm:designs}.
\begin{proof}[Proof of Theorem~\ref{thm:designs}]
We construct a $t$-antidesign in $X=X(m,q)$ and then invoke Lemma~\ref{lem:design_antidesign}. Let $W$ be a fixed $t$-dimensional subspace of $V=V(m,q)$. For $B\in X$ let $B|_W$ be its restriction on $W$ and write
\[
Z=\{B\in X:B|_W=0\}.
\]
We show that $Z$ is a $t$-antidesign. It is plain that $Z$ is additive. By a suitable choice of a basis for $V$, we may assume that the matrices corresponding to the forms in $Z$ range over all symmetric matrices containing a $t\times t$ all-zero submatrix in the upper left corner. From the definition~\eqref{eqn:def_dual} of the dual we readily verify that the matrices corresponding to the forms in $Z^\perp$ must have zero entries in the last $m-t$ rows and columns, forcing their ranks to be at most~$t$. Indeed $\{B|_W:B\in Z^\perp\}$ equals $X(t,q)$. Hence we conclude from Theorem~\ref{thm:dual} that $Z$ is a $t$-antidesign.
\par
Now write
\[
Z'=\{B\in X:\text{$B$ is an extension of $A$}\}.
\]
Then there exists $g\in G$ such that $Z'=gZ$. Conversely, for each $g\in G$, we can choose $U$ and $A$ such that $Z'=gZ$. Since $\abs{Y\cap Z'}$ is the number of forms in $Y$ that are extensions of $A$, the theorem follows from Lemma~\ref{lem:design_antidesign}.
\end{proof}

%%%%%%%%%%%%%%%%%%%%%%%%%%%%%%%%%%%%%%%%%%%%%%%%%%%%%%%%%%%%%%%%%%%%%%%%%%%%%%

\section{Constructions}
\label{sec:constructions}

We say that a $d$-code $Y$ in $X(m,q)$ is \emph{maximal additive} if $Y$ is additive and equality holds in Theorem~\ref{thm:bound_sg}. In this section, we present constructions for maximal additive $d$-codes in $X(m,q)$. Similar constructions have been given already in~\cite{Sch2010}.\footnote{The constructions in~\cite{Sch2010} have been given for finite fields of even characteristic, but as noted in~\cite[Section 5]{Sch2010}, they still work for finite fields of odd characteristic.} Here we pursue a simpler, yet more general, approach.
\par
We shall first show that it is sufficient to construct $d$-codes in $X(m,q)$ for even $m-d$. Let $Y$ be a subset of $X(m,q)$, whose forms are defined on $V=V(m,q)$. Let $W$ be an $(m-1)$-dimensional subspace of $V$, and define the \emph{punctured} set (with respect to $W$) of $Y$ to be
\[
Y^*=\big\{B|_W:B\in Y\big\},
\] 
where $B|_W$ is the restriction of $B$ onto $W$.
\begin{theorem}
\label{thm:code_punctured}
Suppose that $Y$ is a maximal additive $(d+2)$-code in $X(m+1,q)$ for some $d\ge 1$ such that $m-d-1$ is even. Then every punctured set $Y^*$ is a maximal additive $d$-code in $X(m,q)$.
\end{theorem}
\begin{proof}
The theorem follows from Theorem~\ref{thm:bound_sg} and the elementary fact that, if $Y$ is a $(d+2)$-code in $X(m+1,q)$ for some $d\ge 1$, then a punctured set $Y^*$ of $Y$ is a $d$-code in $X(m,q)$.
\end{proof}
\par
The following lemma implies the useful fact that, if $Y$ is an additive $d$-code for some $d\ge 3$ with the property of being a certain design, then each of its punctured versions $Y^*$ also has this design property.
\begin{lemma}
\label{lem:design_punctured}
Let $Y$ be an additive subset of $X(m+1,q)$ and suppose that $Y$ is a $d$-code for some $d\ge 3$. Let $Y^*$ be a punctured version of $Y$ and let $(a'_{k,\e})$ and $(b'_{k,\e})$ be the dual inner distributions of $Y$ and $Y^*$, respectively.~Then
\[
a'_{k,\e}=0\Rightarrow b'_{k,\e}=0.
\]
\end{lemma}
\begin{proof}
We identify symmetric bilinear forms with their matrices and choose the implicit basis such that the matrices in $Y^*$ are $m\times m$ submatrices of the matrices in~$Y$. Since $Y$ is a $d$-code for some $d\ge 3$, we have $\abs{Y}=\abs{Y^*}$. We consider the duals $Y^\perp$ and $(Y^*)^\perp$ of $Y$ and $Y^*$, respectively. Since $Y$ is additive, Theorem~\ref{thm:dual} asserts that $\abs{Y}^{-1}(a'_{k,\e})$ is the inner distribution of $Y^\perp$ and that $\abs{Y}^{-1}(b'_{k,\e})$ is the inner distribution of $(Y^*)^\perp$. Let~$(Y^\perp)^-$ be the ``shortened'' version of $Y^\perp$, namely the subset of $X(m,q)$ obtained by taking those matrices in $Y^\perp$ whose last row (and column) is zero followed by deleting the last row and the last column. It is readily verified that
\[
(Y^*)^\perp=(Y^\perp)^-,
\]
from which the lemma follows.
\end{proof}
\par
Next we provide constructions of $d$-codes in $X(m,q)$ for even $m-d$. We identify $V=V(m,q)$ with $\F_{q^m}$ and use the trace function $\Tr:\F_{q^m}\to\F_q$, given by
\[
\Tr(x)=\sum_{k=0}^{m-1}x^{q^k}.
\]
Let $s$ be a positive integer coprime to $m$, let $t$ be an integer satisfying $0\le t\le (m-1)/2$, and let $\lambda=(\lambda_0,\lambda_1,\dots,\lambda_t)\in V^{t+1}$. Define the symmetric bilinear form
\begin{gather}
B_\lambda:V\times V\to \F_q   \nonumber\\
B_\lambda(x,y)=\Tr\bigg(\lambda_0xy+\sum_{j=1}^t\lambda_j\big(x^{q^{sj}}y+xy^{q^{sj}}\big)\bigg)   \label{eqn:Beven}
\end{gather}
and a subset of $X(m,q)$ by
\[
Y_s(t,m,q)=\big\{B_\lambda:\lambda\in V^{t+1}\big\}.
\]
\par
By the linearity of the trace function, $Y_s(t,m,q)$ and each of its punctured versions are additive. We shall see that $Y_s(t,m,q)$ is an $(m-2t)$-code in $X(m,q)$ of size $q^{m(t+1)}$. Hence, $Y_s(t,m,q)$ (and by Theorem~\ref{thm:code_punctured} also its punctured versions) meet the bounds of Theorem~\ref{thm:bound_sg}. By Lemma~\ref{lem:bound_odd_d}, the set $Y_s(t,m,q)$ is an $(2t+2)$-design for odd $m$. We shall also see that $Y_s(t,m,q)$ is a $(2t+1,\eta(-1)^{t+1})$-design for even~$m$. Therefore Theorem~\ref{thm:inner_dist_odd} and Proposition~\ref{pro:inner_dist_even} give the inner distributions of $Y_s(t,m,q)$ and its punctured versions (using Theorem~\ref{thm:code_punctured} and Lemma~\ref{lem:design_punctured}).
\par
To prove our results, we require the following lemma stated in~\cite[Corollary 1]{BraByrMarMcG2007} for $q=2$, whose proof however works for general $q$ using obvious modifications.
\begin{lemma}
\label{lem:zeros_lin_poly}
Let $m$, $s$, and $d$ be positive integers such that $s$ and $m$ are coprime and let $f\in\F_{q^m}[x]$ be given by
\[
f(x)=\sum_{j=0}^df_jx^{q^{sj}}.
\]
If $f$ is not identically zero, then $f$ has at most $q^d$ zeros in $\F_{q^m}$.
\end{lemma}
\par
\begin{theorem}
\label{thm:Y_code}
$Y_s(t,m,q)$ is an $(m-2t)$-code in $X(m,q)$ of size $q^{m(t+1)}$.
\end{theorem}
\begin{proof}
Since $Y_s(t,m,q)$ is additive, it is sufficient to show that, if $\lambda$ is nonzero, then $B_\lambda$ has rank at least $m-2t$. Pick a nonzero $\lambda\in V^{t+1}$, and note that we can write~\eqref{eqn:Beven} as
\[
B_\lambda(x,y)=\Tr(yL_\lambda(x)),
\]
where
\[
L_\lambda(x)=\lambda_0x+\sum_{j=1}^t\big(\lambda_jx^{q^{sj}}+(\lambda_jx)^{q^{-sj}}\big).
\]
Observe that $B_\lambda(x,y)=0$ for each $y\in V$ if and only if $L_\lambda(x)=0$.  From Lemma~\ref{lem:zeros_lin_poly} we find that $L_\lambda(x^{q^{st}})$ has at most $q^{2t}$ zeros in $V$. Since $x\mapsto x^{q^{st}}$ is an automorphism on $V$, the dimension of the kernel of the $\F_q$-linear mapping $L_\lambda$ is at most $2t$, so that $B_\lambda$ has rank at least $m-2t$, as required.
\end{proof}
\par
In the remainder of this section we show that $Y_s(t,m,q)$ is a $(2t+1,\eta(-1)^{t+1})$-design for even $m$.
\begin{proposition}
\label{pro:Y_design}
For even $m$, the set $Y_s(t,m,q)$ is a $(2t+1,\eta(-1)^{t+1})$-design in $X(m,q)$.
\end{proposition}
\par
We require the following elementary result, generalising~\cite[Lemma~13]{Sch2010}.
\begin{lemma}
\label{lem:bilin_mappings}
Let $U$ be a $k$-dimensional subspace of $V=V(m,q)$. Let $\ell$ be an integer and let $s$ be a positive integer coprime to $m$. Then every bilinear form $B:U\times V\to \F_q$ can be expressed in the form
\[
B(x,y)=\Tr\Bigg(\sum_{j=0}^{k-1}a_jyx^{q^{s(j-\ell)}}\Bigg)
\]
for some uniquely determined $a_0,a_1,\dots,a_{k-1}\in V$.
\end{lemma}
\begin{proof}
Since there are $q^{mk}$ bilinear forms from $U\times V$ to $\F_q$ and the trace is linear, it is enough to show that, if $B(x,y)$ is identically zero, then $a_0=a_1=\cdots=a_{k-1}=0$. If $B(x,y)$ is identically zero, then 
\begin{equation}
\sum_{j=0}^{k-1}a_jx^{q^{s(j-\ell)}}.   \label{eqn:lin_poly_bilinear}
\end{equation}
equals zero for all $x\in U$. It is readily verified using Lemma~\ref{lem:zeros_lin_poly} that, if the $a_j$'s are not all zero, then~\eqref{eqn:lin_poly_bilinear} has at most $q^{k-1}$ zeros in $V$. This completes the proof since $\abs{U}=q^k$.
\end{proof}
\par
We now prove Proposition~\ref{pro:Y_design}.
\begin{proof}[Proof of Proposition~\ref{pro:Y_design}]
We first show that $Y=Y_s(t,m,q)$ is a $(2t+1)$-design. Let $U$ be a $(2t+1)$-dimensional subspace of $V=V(m,q)$ and let $A$ be a symmetric bilinear form on $U$. In view of Theorem~\ref{thm:designs}, we have to show that the number of $B\in Y$ that are extensions of $A$ is a constant independent of the choice of $U$ and $A$.
\par
Let $\mu=(\mu_0,\mu_1,\dots,\mu_{2t})$ be an element of $V^{2t+1}$ and consider the bilinear form $D_\mu:V\times V\to \F_q$ given by 
\[
D_\mu(x,y)=\Tr\Bigg(\sum_{j=0}^{2t}\mu_jyx^{q^{s(j-t)}}\Bigg).
\]
From Lemma~\ref{lem:bilin_mappings} we conclude that, as $\mu$ runs through $V^{2t+1}$, the restriction of $D_\mu$ onto $U\times V$ runs through all bilinear forms on $U\times V$ and hence the restriction of $D_\mu$ onto $U\times U$ equals each bilinear form on $U\times U$ a constant number of times. Now define the symmetric bilinear form $C_\mu:V\times V\to K$ by
\[
C_\mu(x,y)=D_\mu(x,y)+D_\mu(y,x).
\]
Then, as $\mu$ runs through $V^{2t+1}$, the restriction of $C$ onto $U\times U$ equals each symmetric bilinear form on $U\times U$ a constant number of times. An elementary manipulation gives
\begin{equation}
C_\mu(x,y)=\Tr\bigg(\sigma_0xy+\sum_{j=1}^t\sigma_j\big(x^{q^{sj}}y+xy^{q^{sj}}\big)\bigg),   \label{eqn:sym_bilinear_from_D}
\end{equation}
where
\[
\sigma_j=\mu_{t+j}+\mu_{t-j}^{q^{sj}} \quad\text{for each $j\in\{0,1,\dots,t\}$}.
\]
When $\mu$ runs through $V^{2t+1}$, the tuple $(\sigma_0,\sigma_1,\dots,\sigma_t)$ ranges over $V^{t+1}$, where each tuple occurs $q^t$ times. Then, by comparing~\eqref{eqn:sym_bilinear_from_D} and~\eqref{eqn:Beven}, we conclude that the number of $B\in Y$ that are extensions of $A$ is a constant independent of the choice of~$U$ and~$A$. This shows that $Y$ is a $(2t+1)$-design.
\par 
Now let $(a_{i,\t})$ and $(a'_{k,\e})$ be the inner distribution and the dual inner distribution of $Y$, respectively. It remains to show that
\begin{equation}
a'_{2t+1,\e}=0\quad\text{for $\e=\eta(-1)^{t+1}$}.   \label{eqn:a2t1_zero}
\end{equation}
Let $A_s,B_s$ and $A'_r,C'_r$ be as defined in~\eqref{eqn:def_ABCD} and~\eqref{eqn:def_ABCD_dual}, respectively, and write $n=m/2$. From Lemma~\ref{lem:four_equations} and~\eqref{eqn:ev_transform} we find that
\begin{equation}
\sum_{r=0}^{t+1}{n-r\brack n-t-1} \big(A'_r+q^{t+1}\; C'_r\big)=q^{m(t+1)}\sum_{s=0}^{n}{n-s\brack t+1}\big(q^{t+1}\,A_s+B_s\big).   \label{eqn:eqn:AC}
\end{equation}
Since $Y$ is a $(2n-2t)$-code by Theorem~\ref{thm:Y_code} and $A_0=B_0=1$, the right-hand side of~\eqref{eqn:eqn:AC} reduces to
\[
q^{m(t+1)}{n\brack t+1}(q^{t+1}+1).
\]
Since $A'_0=C'_0=\abs{Y}$ and $Y$ is a $(2t+1)$-design, the left-hand side of~\eqref{eqn:eqn:AC} is
\[
\abs{Y}\,{n\brack t+1} \big(q^{t+1}+1\big)+A'_{t+1}+q^{t+1}C'_{t+1}.
\]
By Theorem~\ref{thm:Y_code}, we have $\abs{Y}=q^{m(t+1)}$ and therefore
\[
A'_{t+1}+q^{t+1}C'_{t+1}=0.
\]
Since $a'_{2t+1,1}=a'_{2t+1,-1}=0$, this is equivalent to~\eqref{eqn:a2t1_zero}, which completes the proof.
\end{proof}

%%%%%%%%%%%%%%%%%%%%%%%%%%%%%%%%%%%%%%%%%%%%%%%%%%%%%%%%%%%%%%%%%%%%%%%%%%%%%%

\section{Applications to classical coding theory}
\label{sec:coding}

We begin with recalling some basic facts from coding theory (see~\cite{MacSlo1977} for more background). A \emph{code} over $\F_q$ of \emph{length} $n$ is a subset of $\F_q^n$; such a code is \emph{additive} if it is a subgroup of $(\F_q^n,+)$. The (Hamming) \emph{weight} of $c\in\F_q^n$, denoted by $\wt(c)$, is the number of nonzero entries in $c$. This weight induces a distance on $\F_q^n$ and the smallest distance between two distinct elements of a code $\C$ is called the \emph{minimum distance} of $\C$. Given a code $\C$, the polynomials
\[
\alpha(z)=\sum_{c\in\C}z^{\wt(c)}
\]
and
\[
\beta(z)=\frac{1}{\abs{\C}}\sum_{b,c\in\C}z^{\wt(c-b)}
\]
are called the \emph{weight enumerator} and the \emph{distance enumerator} of $\C$, respectively. Note that, if~$\C$ is additive, then its weight enumerator coincides with its distance enumerator. \par
We are going to study codes that are derived from quadratic forms. Consider $V=V(m,q)$ and write $V^*=V-\{0\}$. Every symmetric bilinear form $B$ in $X(m,q)$ gives rise to a quadratic form $Q:V\to\F_q$ via $Q(x)=B(x,x)$. Conversely, every quadratic form $Q:V\to\F_q$ gives a bilinear form $B$ in $X(m,q)$ via
\[
B(x,y)=\frac{1}{2}\big(Q(x+y)-Q(x)-Q(y)\big).
\]
Accordingly, the \emph{rank} and \emph{type} of a quadratic form is defined to be the rank and type of the corresponding bilinear form, respectively. Since for every quadratic form $Q:V\to\F_q$ we have $Q(0)=0$, we can without loss of generality restrict $Q$ to $V^*$. We shall identify functions from $V^*$ to $\F_q$ as vectors in~$\F_q^{V^*}$.
\par
It is well known (see~\cite{FenLuo2008}, for example) that the weight of a quadratic form $Q$ on $V$ is intimately related to an exponential sum of type
\[
\sum_{x\in V}\chi(Q(x)),
\]
where $\chi$ is a character of $(\F_q,+)$. This sum depends only on the character $\chi$, and the rank and type of the quadratic form $Q$. Thus the analysis of codes derived from sets of quadratic forms naturally leads to the study of the inner distributions of the underlying subsets of $X(m,q)$.
\par
The distance enumerators of codes derived from quadratic forms have been studied extensively in the last years (see~\cite{FenLuo2008}, \cite{LuoFen2008}, \cite{LiuLuo2014}, \cite{LiuLuo2014a}, \cite{ZenHuJiaYueCao2010}, \cite{ZheWanZenHu2013}, \cite{LuoTanWan2010}, \cite{LiuYanLiu2014} for example). These contributions are based on an ad hoc analysis of the underlying quadratic or bilinear forms and their corresponding character sums. This method becomes quickly unfeasible for larger codes; the current results apply only for codes of size $\abs{V}^s$, where $s\le 4$. In this section we offer a unified treatment that is much simpler and yet more general.
\par
Let $Y$ be a subset of $X=X(m,q)$. We shall associate with $Y$ two codes of length $q^m-1$. We define the code $\C_1(Y)$ to be the set of vectors in $\F_q^{V^*}$ containing the quadratic forms corresponding to $Y$. Denote by $\L=\L(m,q)$ the set of vectors in $\F_q^{V^*}$ corresponding to the linear functions from $V$ to $\F_q$. This is an additive code of length $q^m-1$ and cardinality $q^m$. Its lengthened version is known as the \emph{generalised first-order Reed-Muller code}~(see~\cite{DelGoeMac1970}, for example). We define the code $\C_2(Y)$ to be the union of cosets of $\L$ with coset representatives coming from $\C_1(Y)$, that is
\begin{equation}
\C_2(Y)=\{a+b:a\in \L,\,b\in\C_1(Y)\}.   \label{eqn:def_C2}
\end{equation}
We note that, if $Y$ is additive, then $\C_1(Y)$ and $\C_2(Y)$ are additive codes.
\par
In order to find the distance enumerators of $\C_1(Y)$ and $\C_2(Y)$, we require the following standard result on quadratic forms (see~\cite[Section~6.2]{LidNie1997}, for example). Note that we continue using $\eta$ to denote the quadratic character of $\F_q$ and extend $\eta$ in the standard way by putting $\eta(0)=0$.
\begin{lemma}
\label{lem:quad_forms_sol}
Let $Q:V(m,q)\to\F_q$ be a quadratic form of rank $i$ and type~$\t$. Let $N(h)$ be the number of solutions $x\in V(m,q)$ of $Q(x)=h$. Then
\[
N(h)=
\begin{cases}
q^{m-1}+\t\,\eta(-1)^{(i-1)/2}\,\eta(h)\,q^{m-(i+1)/2} & \text{for odd $i$}\\[1ex]
q^{m-1}+\t\,\eta(-1)^{i/2}\,v(h)\,q^{m-i/2-1}          & \text{for even $i$},
\end{cases}
\] 
where $v(x)=-1$ for every nonzero $x\in\F_q$ and $v(0)=q-1$.
\end{lemma}
\par
We immediately obtain the following result from Lemma~\ref{lem:quad_forms_sol} applied with $h=0$.
\begin{theorem}
\label{thm:dist_C1}
Let $Y$ be a subset of $X(m,q)$ with inner distribution $(a_{i,\t})$. Then the distance enumerator of $\C_1(Y)$ is $\sum_{i,\t}a_{i,\t}z^{w_{i,\t}}$, where
\[
w_{i,\t}=\begin{cases}
q^{m-1}(q-1)                           & \text{for odd $i$},\\[1ex]
(q^{m-1}-\t\eta(-1)^{i/2}\,q^{m-i/2-1})(q-1) & \text{for even $i$}.
\end{cases}
\]
\end{theorem}
\par
In order to obtain a corresponding result for $\C_2(Y)$, we define the polynomial
\[
T_{i,\t}(z)=n_u z^u+n_v z^v +n_w z^w,
\]
where
\begin{align*}
u&=q^{m-1}(q-1)-q^{m-(i+1)/2}  & n_u&=\tfrac12(q^{i-1}+q^{(i-1)/2})(q-1) \\
v&=q^{m-1}(q-1)                & n_v&=q^m-q^{i-1}(q-1)\\
w&=q^{m-1}(q-1)+q^{m-(i+1)/2}  & n_w&=\tfrac12(q^{i-1}-q^{(i-1)/2})(q-1)
\intertext{for odd $i$ and}
u&=(q^{m-1}-\tau\,\eta(-1)^{i/2}\, q^{m-i/2-1})(q-1)  & n_u&=q^{i-1}+\tau\,\eta(-1)^{i/2}\, q^{i/2-1}(q-1) \\
v&=q^{m-1}(q-1)                                       & n_v&=q^m-q^i\\
w&=q^{m-1}(q-1)+\tau\,\eta(-1)^{i/2}\, q^{m-i/2-1}    & n_w&=(q^{i-1}-\tau\,\eta(-1)^{i/2}\, q^{i/2-1})(q-1)
\end{align*}
for even $i$. In fact, $T_{i,\t}(z)$ is a binomial for $i=0$ or $1$, otherwise it is a trinomial. The reason for this definition of $T_{i,\t}(z)$ becomes clear from the following lemma.
\begin{lemma}
\label{lem:coset_weight_distribution}
Let $Q:V(m,q)\to\F_q$ be a quadratic form of rank $i$ and type~$\t$. Then $T_{i,\t}(z)$ is the weight enumerator of the coset $Q+\L$.
\end{lemma}
\begin{proof}
By choosing an appropriate basis for $V(m,q)$ over $\F_q$, we can write~\cite[Theorem~6.21]{LidNie1997}
\[
Q(x)=\sum_{j=1}^ia_jx_j^2
\]
for some nonzero $a_1,\dots,a_i\in\F_q$ satisfying $\eta(a_1\cdots a_i)=\tau$, where $x_1,\dots,x_m$ are the coordinates of $x$ relative to the basis chosen. Thus, writing
\[
f(x_1,\dots,x_m)=\sum_{j=1}^ia_jx_j^2+\sum_{j=1}^m b_jx_j
\]
and noting that $f(0,\dots,0)=0$, we have to show that
\[
T_{i,\t}(z)=\sum_{b_1,\dots,b_m\in\F_q}z^{\wt(f)}.
\]
If one of $b_{i+1},\dots,b_m$ is nonzero, which occurs in $q^m-q^i$ cases, then it is readily verified that $\wt(f)$ equals $q^m-q^{m-1}$. It therefore remains to consider the cases that $b_{i+1},\dots,b_m$ are all zero. 
\par
We introduce the new variable $y_j$ by setting $x_j=y_j-b_j/(2a_j)$ for each $j\in\{1,\dots,m\}$. Then the number of solutions $(x_1,\dots,x_m)\in\F_q^m$ of $f(x_1,\dots,x_m)=0$ equals the number of solutions $(y_1,\dots,y_m)\in\F_q^m$ of
\begin{equation}
\sum_{j=1}^ia_jy_j^2=h,   \label{eqn:diag_eqn_h}
\end{equation}
where
\begin{equation}
h=\sum_{j=1}^i\frac{1}{4a_j}\,b_j^2.   \label{eqn:h}
\end{equation}
Notice that $h$ is a quadratic form in $b_1,\dots,b_i$ of rank $i$ and type $\tau$. We treat the left-hand side of~\eqref{eqn:diag_eqn_h} as a quadratic form in $y_1,\dots,y_m$ (again of rank $i$ and type~$\tau$). We now distinguish the cases that $i$ is odd and $i$ is even, making multiple use of Lemma~\ref{lem:quad_forms_sol} without explicit reference.
\par
First assume that $i$ is odd. There are $q^{i-1}$ tuples $(b_1,\dots,b_i)$ such that $h=0$, in which case the number of solutions $(y_1,\dots,y_m)\in\F_q^m$ of~\eqref{eqn:diag_eqn_h} is $q^{m-1}$. Combining this with the contributions arising when at least one of $b_{i+1},\dots,b_m$ is nonzero, we find that the coefficient of term of degree $q^m-q^{m-1}$ in $T_{i,\t}(z)$ equals
\[
q^m-q^i+q^{i-1},
\]
as required. The number of tuples $(b_1,\dots,b_i)$ satisfying~\eqref{eqn:h} for a fixed nonzero $h$ equals
\[
q^{i-1}+\tau\,\eta(-1)^{(i-1)/2}\,\eta(h)\, q^{(i-1)/2}.
\]
Since $\F_q$ contains $(q-1)/2$ squares and the same number of nonsquares, we find that, for each $\nu\in\{-1,1\}$, the number of tuples $(b_1,\dots,b_i)$ satisfying~\eqref{eqn:h} for some $h$ with $\eta(h)=\nu$ is
\[
\tfrac12(q-1)(q^{i-1}+\nu\,\tau\,\eta(-1)^{(i-1)/2}\, q^{(i-1)/2}).
\]
The corresponding number of solutions to~\eqref{eqn:diag_eqn_h} is 
\[
q^{m-1}+\nu\,\tau\,\eta(-1)^{(i-1)/2}\, q^{m-(i+1)/2}.
\]
Therefore, the coefficient of the term of degree 
\[
q^m-(q^{m-1}\pm q^{m-(i+1)/2})
\]
in $T_{i,\t}(z)$ equals
\[
\tfrac12(q-1)(q^{i-1}\pm q^{(i-1)/2}),
\]
as required.
\par
Now assume that $i$ is even. We treat this case in less detail. The number of $(b_1,\dots,b_i)\in\F_q^m$ satisfying~\eqref{eqn:h} for $h=0$ equals
\[
q^{i-1}+\tau\,\eta(-1)^{i/2}(q-1)q^{i/2-1}
\]
and the corresponding number of solutions to~\eqref{eqn:diag_eqn_h} is 
\[
q^{m-1}+\tau\,\eta(-1)^{i/2}(q-1)q^{m-i/2-1}.
\]
The number of $(b_1,\dots,b_i)\in\F_q^m$ satisfying~\eqref{eqn:h} for some nonzero $h\in\F_q$ equals
\[
(q-1)(q^{i-1}-\tau\,\eta(-1)^{i/2}\,q^{i/2-1})
\]
and the corresponding number of solutions to~\eqref{eqn:diag_eqn_h} is 
\[
q^{m-1}-\tau\,\eta(-1)^{i/2}\,q^{m-i/2-1}.
\]
From this we readily verify the statement of the lemma for even $i$.
\end{proof}
\par
We are now ready to prove the following theorem for $\C_2(Y)$.
\begin{theorem}
\label{thm:dist_C2}
Let $Y$ be a subset of $X(m,q)$ with inner distribution $(a_{i,\t})$. Then the distance enumerator of $\C_2(Y)$ is $\sum_{i,\t}a_{i,\t}T_{i,\t}(z)$. 
\end{theorem}
\begin{proof}
By the definition~\eqref{eqn:def_C2} of $\C_2(Y)$, the distance enumerator of $\C_2(Y)$ equals
\[
\frac{1}{\abs{Y}}\sum_{b,c\in\C_1(Y)}\;\sum_{a\in\L}z^{\wt(a+c-b)},
\]
using that $\L$ is additive. The inner sum is just the weight enumerator of the coset $c-b+\L$ and so the theorem follows from Lemma~\ref{lem:coset_weight_distribution}.
\end{proof}
\par
Now consider the codes obtained from the sets $Y=Y_s(t,m,q)$ constructed in Section~\ref{sec:constructions} and their punctured versions $Y^*=Y_s(t,m+1,q)^*$. Since Theorem~\ref{thm:inner_dist_odd} and Proposition~\ref{pro:inner_dist_even} give the inner distributions of $Y$ and $Y^*$, Theorems~\ref{thm:dist_C1} and~\ref{thm:dist_C2} give the distance enumerators of the derived code families. In particular, the minimum distance of $\C_1(Y)$ is $$(q^{m-1}-q^{\floor{m/2}+t-1})(q-1)$$ and that of $\C_2(Y)$ is
\[
\begin{cases}
q^{m-1}(q-1)-q^{(m-1)/2+t}  & \text{for odd $m$}\\
(q^{m-1}-q^{m/2+t-1})(q-1)  & \text{for even $m$}.
\end{cases}
\] 
The minimum distance of $\C_1(Y^*)$ is $$(q^{m-1}-q^{\floor{(m-1)/2}+t})(q-1)$$ and that of $\C_2(Y^*)$ is
\[
\begin{cases}
(q^{m-1}-q^{(m-1)/2+t})(q-1)  & \text{for odd $m$}\\
q^{m-1}(q-1)-q^{m/2+t}  & \text{for even $m$}.
\end{cases}
\]
\par
The codes obtained from $Y=Y_s(t,m,q)$ have received considerable attention in the literature and are discussed next in more detail. Consider the bilinear form $B_\lambda$ given in~\eqref{eqn:Beven}. The corresponding quadratic form is given by
\begin{gather*}
Q_\lambda:\F_{q^m}\to\F_q\\
Q_\lambda(x)=\Tr_m\bigg(\lambda_0x^2+2\sum_{j=1}^t\lambda_jx^{q^{sj}+1}\bigg).
\end{gather*}
The codewords of $\C_1(Y)$ can be written explicitly as
\[
\big(Q_\lambda(1),\,Q_\lambda(\theta),\,Q_\lambda(\theta^2),\dots,\,Q_\lambda(\theta^{q^m-2})\big),
\]
where $\theta$ is a primitive element of $\F_{q^m}$. Thus $\C_1(Y)$ is a cyclic code, whose zeros are 
\[
\theta^{-2},\theta^{-(q^s+1)},\theta^{-(q^{2s}+1)},\dots,\theta^{-(q^{ts}+1)},
\] 
by Delsarte's theorem~\cite[Ch.~7, Theorem~11]{MacSlo1977}. In the same way, the code $\C_2(Y)$ can also be made cyclic with zeros
\[
\theta^{-1},\theta^{-2},\theta^{-(q^s+1)},\theta^{-(q^{2s+}1)},\dots,\theta^{-(q^{ts}+1)}.
\]
In several special cases, the distance enumerators of $\C_1(Y)$ and $\C_2(Y)$ have been obtained previously using long calculations. Those of $\C_1(Y_s(1,q,m))$ and $\C_2(Y_s(1,q,m))$ have been determined in~\cite{FenLuo2008} for prime $q$ and in~\cite{LuoFen2008} for general $q$. Moreover,~\cite{LiuLuo2014} gives the distance enumerator of $\C_1(Y_1(2,3,m))$ for odd $m$ and the distance enumerator of $\C_2(Y_1(2,3,m))$ for $m\equiv 1\pmod 4$ and~\cite{LiuLuo2014a} gives the distance enumerator of $\C_1(Y_1(2,3,m))$ for even $m$. For prime $p$ and odd $m$, the distance enumerator of $\C_1(Y_s(2,p,m))$ has been determined in~\cite{LiuHan2013} (and in disguise in~\cite{ZheWanZenHu2013}) and that of $\C_2(Y_s(2,p,m))$ has been determined in~\cite{LiuYanLiu2014}. 
\par
The authors of~\cite{ZenHuJiaYueCao2010} study codes with the same distance enumerator as $\C_2(Y_s(1,p,m))$ for prime $p$ and odd $m$. These results can also be easily reproduced using the methods of this paper.
\par
Our methods also simplify the analysis of other code families derived from quadratic forms, such as those codes considered in~\cite{LuoFen2008a} or~\cite{ZhoDinLuoZha2013}, and are useful to study sequence sets derived from quadratic forms, such as those studied in~\cite{KumMor1991}.

%%%%%%%%%%%%%%%%%%%%%%%%%%%%%%%%%%%%%%%%%%%%%%%%%%%%%%%%%%%%%%%%%%%%%%%%%%%%%%

\appendix

\section{Computation of the $Q$-numbers}

We shall now prove Theorem~\ref{thm:eigenvalues}. We write $Q^{(m)}_{k,\e}(i,\t)$ for $Q_{k,\e}(i,\t)$ to indicate dependence on $m$. It will also be convenient to write
\begin{equation}
Q^{(m)}_{0,-1}(i,\t)=0.   \label{eqn:initial_0}
\end{equation}
We begin with the following recurrence for the $Q$-numbers.
\begin{lemma}
\label{lem:rec_P}
For $k,i\ge 1$, we have
\begin{multline*}
Q^{(m)}_{k,\e}(i,\t)=Q^{(m)}_{k,\e}(i-1,1)\\[1ex]
-\tfrac{1}{2}q^{m-i}\gamma_q^{i-1}
\Big[Q^{(m-1)}_{k-1,\e}(i-1,1)(q-\t\,\gamma_q)-(-1)^iQ^{(m-1)}_{k-1,-\e}(i-1,1)(q+\t\,\gamma_q)\Big].
\end{multline*}
\end{lemma}
\begin{proof}
Let $X^{(m)}_{i,\t}$ be the set of $m\times m$ symmetric matrices of rank $i$ and type $\t$. Fix $i\in\{1,\dots,m\}$ and $\t\in\{-1,1\}$. Let $S$ be an $m\times m$ diagonal matrix of rank $i$ with diagonal $[z,1,\dots,1,0,\dots,0]$ such that $\eta(z)=\t$, and let $S'$ be the $(m-1)\times(m-1)$ diagonal matrix of rank $i-1$ with diagonal $[1,\dots,1,0,\dots,0]$. From~\eqref{eqn:Q_from_P} and~\eqref{eqn:P_numbers} we have
\begin{align}
Q^{(m)}_{k,\e}(i-1,1)-Q^{(m)}_{k,\e}(i,\t)&=\sum_{A\in X^{(m)}_{k,\e}}\big(\ang{B,S'}-\ang{A,S}\big)   \label{eqn:diff_P}\\
&=\sum_{A\in X^{(m)}_{k,\e}}\ang{B,S'}\big(1-\chi(za)\big),   \label{eqn:sum_AB}
\end{align}
where we write $A$ as
\begin{equation}
A=\begin{bmatrix}
a & v^T\\
v & B
\end{bmatrix}   \label{eqn:def_A}
\end{equation}
for some $a\in\F_q$, some $v\in\F_q^{m-1}$, and some $(m-1)\times(m-1)$ symmetric matrix $B$ over $\F_q$. The summand in~\eqref{eqn:sum_AB} is zero for $a=0$, so assume that $a$ is nonzero. Writing
\[
L=\begin{bmatrix}
1 & -a^{-1}v^T\\
0 & I
\end{bmatrix},
\]
we have
\[
L^TAL=\begin{bmatrix}
a & 0\\
0 & C
\end{bmatrix},\quad\text{where}\quad C=B-a^{-1}vv^T.
\]
Note that $L$ is nonsingular. Let $\mathcal{Q}$ and $\mathcal{N}$ be the set of squares and nonsquares of $\F_q^*$, respectively. As $(a,C)$ ranges over
\[
\mathcal{Q} \times X^{(m-1)}_{k-1,\e} \;\; \cup \;\; \mathcal{N}\times X^{(m-1)}_{k-1,-\e}
\]
and $v$ ranges over $\F_q^{m-1}$, the matrix $A$, given in~\eqref{eqn:def_A}, ranges over all elements of $X^{(m)}_{k,\e}$, except for those matrices~\eqref{eqn:def_A} satisfying $a=0$. Hence, using~\eqref{eqn:ip_homomorphism}, the sum~\eqref{eqn:sum_AB} becomes
\begin{equation}
\begin{split}
&\sum_{a\in \mathcal{Q}}\big(1-\chi(za)\big)\sum_{C\in X^{(m-1)}_{k-1,\e}}\ang{C,S'}\sum_{v\in\F_q^{m-1}}\ang{a^{-1}vv^T,S'}\\
+&\sum_{a\in \mathcal{N}}\big(1-\chi(za)\big)\sum_{C\in X^{(m-1)}_{k-1,-\e}}\ang{C,S'}\sum_{v\in\F_q^{m-1}}\ang{a^{-1}vv^T,S'}.
\end{split}
\label{eqn:sum_AB_2}
\end{equation}
From~\eqref{eqn:Q_from_P} and~\eqref{eqn:P_numbers} we have
\begin{equation}
\sum_{C\in X^{(m-1)}_{k-1,\e}}\ang{C,S'}=Q^{(m-1)}_{k-1,\e}(i-1,1).   \label{eqn:Pmki_minus_one}
\end{equation}
Furthermore,
\begin{align}
\sum_{v\in\F_q^{m-1}}\ang{a^{-1}vv^T,S'}&=q^{m-i}\sum_{v_1,\dots,v_{i-1}\in\F_q}\chi(a^{-1}(v_1^2+\cdots+v_{i-1}^2))   \nonumber\\
&=q^{m-i}\Bigg(\sum_{v\in\F_q}\chi(a^{-1}v^2)\Bigg)^{i-1}.   \label{eqn:sum_v}
\end{align}
Putting $\eta(0)=0$, the summation becomes
\begin{align*}
\sum_{v\in\F_q}\chi(a^{-1}v^2)&=\sum_{y\in\F_q}(1+\eta(y))\chi(a^{-1}y)\\
&=\sum_{y\in\F_q}(1+\eta(ay))\chi(y)\\
&=\eta(a)\,\gamma_q,
\end{align*}
using $\sum_{y\in\F_q}\chi(y)=0$ and the definition~\eqref{eqn:def_gauss_sum} of the Gauss sum~$\gamma_q$. Substitute into~\eqref{eqn:sum_v} and then~\eqref{eqn:sum_v} and~\eqref{eqn:Pmki_minus_one} into~\eqref{eqn:sum_AB_2} to deduce that~\eqref{eqn:diff_P} equals
\[
q^{m-i}\gamma_q^{i-1}\Big[Q^{(m-1)}_{k-1,\e}(i-1,1)\sum_{a\in \mathcal{Q}}\big(1-\chi(za)\big)-(-1)^iQ^{(m-1)}_{k-1,-\e}(i-1,1)\sum_{a\in \mathcal{N}}\big(1-\chi(za)\big)\Big].
\]
To complete the proof, recall that $\abs{\mathcal{Q}}=\abs{\mathcal{N}}=(q-1)/2$ and observe that
\begin{align*}
\sum_{a\in \mathcal{Q}}\big(1-\chi(za)\big)&=\frac{q-1}{2}-\sum_{a\in\F_q^*}\frac{1+\eta(a)}{2}\chi(za)\\
&=\frac{q}{2}-\frac{1}{2}\sum_{a\in\F_q^*}\eta(a)\chi(za)\\
&=\frac{q}{2}-\frac{1}{2}\eta(z)\,\gamma_q
\intertext{and similarly}
\sum_{a\in \mathcal{N}}\big(1-\chi(za)\big)&=\frac{q}{2}+\frac{1}{2}\eta(z)\,\gamma_q,
\end{align*}
as required.
\end{proof}
\par
The $Q$-numbers of $X(m,q)$ are determined by the recurrence of Lemma~\ref{lem:rec_P} together with the initial numbers $Q^{(m)}_{k,\e}(0,1)$ and $Q^{(m)}_{0,1}(i,\t)$. From~\eqref{eqn:Q_from_P} and~\eqref{eqn:P_numbers} we find that these initial numbers satisfy
\begin{align}
Q^{(m)}_{0,1}(i,\t)&=1,   \label{eqn:Q_initial_1}\\
Q^{(m)}_{k,\e}(0,1)&=v^{(m)}(k,\e),   \label{eqn:Q_initial_2}
\end{align}
where we write $v^{(m)}(i,\t)$ for $v(i,\t)$, the number of $m\times m$ symmetric matrices over $\F_q$ of rank $i$ and type $\t$. 
\par
In order to solve the recurrence in Lemma~\ref{lem:rec_P}, we shall first obtain explicit expressions for the numbers
\begin{align}
S^{(m)}_k(i,\t)&=Q^{(m)}_{k,1}(i,\t)+Q^{(m)}_{k,-1}(i,\t),   \label{eqn:def_S}\\[1ex]
R^{(m)}_k(i,\t)&=Q^{(m)}_{k,1}(i,\t)-Q^{(m)}_{k,-1}(i,\t)   \label{eqn:def_R}
\end{align}
in the following two lemmas. Since
\[
2Q^{(m)}_{k,\e}(i,\t)=S^{(m)}_k(i,\t)+\e\,R^{(m)}_{k}(i,\t)
\]
we then directly obtain Theorem~\ref{thm:eigenvalues}.  We shall repeatedly use the following elementary identity for the Gauss sum $\gamma_q$, defined in~\eqref{eqn:def_gauss_sum},
\begin{equation}
\gamma_q^2=\eta(-1)q   \label{eqn:gauss_square}
\end{equation}
(see~\cite[Theorem~5.12]{LidNie1997}, for example).
\begin{lemma}
\label{lem:S}
For $k,i\ge 1$, the numbers $S^{(m)}_k(i,\t)$ are given by
\begin{align}
S^{(m)}_{2r+1}(2s+1,\t)&=-q^{2r}F^{(m-1)}_r(s),   \label{eqn:S1}\\[1.5ex]
S^{(m)}_{2r}(2s+1,\t)&=q^{2r}F^{(m-1)}_r(s),   \label{eqn:S2}\\[1.5ex]
S^{(m)}_{2r+1}(2s,\t)&=-q^{2r}F^{(m-1)}_r(s-1)+\t\,\eta(-1)^sq^{m-s+2r}F^{(m-2)}_r(s-1),   \label{eqn:S3}\\[1.5ex]
S^{(m)}_{2r}(2s,\t)&=q^{2r}F^{(m-1)}_r(s-1)-\t\,\eta(-1)^sq^{m-s+2r-2}F^{(m-2)}_{r-1}(s-1).   \label{eqn:S4}
\end{align}
\end{lemma}
\begin{proof}
Write
$s=\lfloor i/2\rfloor$ and
\[
n=\floor{(m-1)/2}\quad\text{and}\quad c=q^{(m-1)(m-2)/(2n)}.
\]
From Lemma~\ref{lem:rec_P} we find that
\begin{equation}
S^{(m)}_k(i,\t)=S^{(m)}_k(i-1,1)-(-\t)^{i+1}\,q^{m-2s}\gamma_q^{2s}\,S^{(m-1)}_{k-1}(i-1,1).   \label{eqn:rec_S}
\end{equation}
Hence, by~\eqref{eqn:Q_initial_2} and~\eqref{eqn:def_S},
\[
S_k^{(m)}(1,\t)=v^{(m)}(k,1)+v^{(m)}(k,-1)-q^m\big[v^{(m-1)}(k-1,1)+v^{(m-1)}(k-1,-1)\big].
\]
From Proposition~\ref{pro:num_sym_mat} we then find that
\[
S^{(m)}_{2r}(1,\t)=-S^{(m)}_{2r+1}(1,\t)=\frac{1}{q^r}\,\frac{(q^m-q)(q^{m}-q^2)\cdots(q^m-q^{2r})}{(q^{2r}-1)(q^{2r}-q^2)\cdots(q^{2r}-q^{2r-2})},
\]
which we can write as
\[
q^{2r}{n\brack r}\prod_{j=0}^{r-1}(c-q^{2j}).
\]
Apply the identities~\eqref{eqn:cross_identity} and~\eqref{eqn:q-binomial} to $F^{(m-1)}_r(0)$ to see that~\eqref{eqn:S1} and~\eqref{eqn:S2} hold for $s=0$. Moreover, it is easy to see from~\eqref{eqn:initial_0} and~\eqref {eqn:Q_initial_1} that~\eqref{eqn:S2} also holds for $r=0$. Now substitute the recurrence~\eqref{eqn:rec_S} into itself and use~\eqref{eqn:gauss_square} to obtain
\[
S^{(m)}_k(2s+1,\t)=S^{(m)}_k(2s-1,1)-cq^{2(n-s+1)}\,S^{(m-2)}_{k-2}(2s-1,1).
\]
Using~\eqref{eqn:Pascal_triangle}, we readily verify by induction that~\eqref{eqn:S1} and~\eqref{eqn:S2} hold for all $r,s\ge 0$. The identities~\eqref{eqn:S3} and~\eqref{eqn:S4} then follow from~\eqref{eqn:rec_S} and~\eqref{eqn:gauss_square}.
\end{proof}
\par
\begin{lemma}
For $k,i\ge 1$, the numbers $R^{(m)}_k(i,\t)$ are given by
\begin{align}
R^{(m)}_{2r+1}(2s+1,\t)&=\t\,\eta(-1)^{r+s}\, q^{m-s+r-1}\,\gamma_q\,F^{(m-1)}_r(s),   \label{eqn:R1}\\[1.5ex]
R^{(m)}_{2r}(2s+1,\t)&=\eta(-1)^rq^rF^{(m)}_r(s),   \label{eqn:R2}\\[1.5ex]
R^{(m)}_{2r+1}(2s,\t)&=0,   \label{eqn:R3}\\[1.5ex]
R^{(m)}_{2r}(2s,\t)&=\eta(-1)^rq^rF^{(m)}_r(s).   \label{eqn:R4}
\end{align}
\end{lemma}
\begin{proof}
Write
\[
n=\floor{m/2}\quad\text{and}\quad c=q^{m(m-1)/(2n)}.
\]
From~\eqref{eqn:Q_initial_2} and~\eqref{eqn:def_R} we have
\[
R^{(m)}_k(0,1)=v^{(m)}(k,1)-v^{(m)}(k,-1).
\]
From Proposition~\ref{pro:num_sym_mat} we then find that
\[
R^{(m)}_{2r+1}(0,1)=0
\]
and
\begin{align*}
R^{(m)}_{2r}(0,1)&=\eta(-1)^rq^r{n\brack r}\prod_{j=0}^{r-1}(c-q^{2j})   \nonumber\\[1ex]
&=\eta(-1)^rq^rF^{(m)}_r(0),
\end{align*}
as in the proof of Lemma~\ref{lem:S}. Hence~\eqref{eqn:R3} and~\eqref{eqn:R4} hold for $s=0$. Moreover~\eqref{eqn:R4} trivially holds for $r=0$. Writing $s=\lfloor (i+1)/2\rfloor$, we find from Lemma~\ref{lem:rec_P} that
\begin{equation}
R^{(m)}_k(i,\t)=R^{(m)}_k(i-1,1)-(-\t)^i\,q^{m-2s+1}\gamma_q^{2s-1}\,R^{(m-1)}_{k-1}(i-1,1),   \label{eqn:rec_R}
\end{equation}
which, after self-substitution, gives
\[
R^{(m)}_k(2s,\d)=R^{(m)}_k(2s-2,1)-\eta(-1)\,qc\,q^{2(n-s)}\,R^{(m-2)}_{k-2}(2s-2,1),
\]
using~\eqref{eqn:gauss_square}. With this recurrence,~\eqref{eqn:R3} immediately follows and~\eqref{eqn:R4} can be verified by induction, as in the proof of Lemma~\ref{lem:S}. The identities~\eqref{eqn:R1} and~\eqref{eqn:R2} then follow from~\eqref{eqn:rec_R} and~\eqref{eqn:gauss_square}.
\end{proof}

%%%%%%%%%%%%%%%%%%%%%%%%%%%%%%%%%%%%%%%%%%%%%%%%%%%%%%%%%%%%%%%%%%%%%%%%%%%%%%

\section*{Acknowledgements}

I would like to thank Rod Gow and Tor Helleseth for fruitful discussions, which inspired some of the results in this paper. In particular, Rod Gow proved independently with a different technique that Theorem~\ref{thm:bound_sg} holds for $d=m-1$ and $d=m-2$.

%%%%%%%%%%%%%%%%%%%%%%%%%%%%%%%%%%%%%%%%%%%%%%%%%%%%%%%%%%%%%%%%%%%%%%%%%%%%%%

% \bibliographystyle{plain}
% \bibliography{references}

\end{document}